\documentclass[a4paper,10pt]{article}
\usepackage[utf8]{inputenc}
\usepackage{gensymb}
\usepackage[square,numbers]{natbib}
\usepackage{amssymb}
\usepackage{amsfonts}
\usepackage{latexsym}
\usepackage{amsmath,systeme}
\usepackage{graphics}
\usepackage{subfigure}
\usepackage{graphicx}
\usepackage{epsfig}
\usepackage[leqno]{mathtools}
\usepackage{amsthm}
\usepackage[titletoc]{appendix}
\usepackage[dvipsnames,usenames]{xcolor}

\mathtoolsset{showonlyrefs}

\numberwithin{equation}{section}

\newtheorem{theorem}{Theorem}[section]

\newtheorem{lemma}[theorem]{Lemma}
\newtheorem{definition}[theorem]{Definition}

\newtheorem{proposition}[theorem]{Proposition}

\newtheorem{remark}{Remark}

\usepackage{esint}

\def\vint_#1{\mathchoice%
          {\mathop{\kern 0.2em\vrule width 0.6em height 0.69678ex depth -0.58065ex
                  \kern -0.8em \intop}\nolimits_{\kern -0.4em#1}}%
          {\mathop{\kern 0.1em\vrule width 0.5em height 0.69678ex depth -0.60387ex
                  \kern -0.6em \intop}\nolimits_{#1}}%
          {\mathop{\kern 0.1em\vrule width 0.5em height 0.69678ex
              depth -0.60387ex
                  \kern -0.6em \intop}\nolimits_{#1}}%
          {\mathop{\kern 0.1em\vrule width 0.5em height 0.69678ex depth -0.60387ex
                  \kern -0.6em \intop}\nolimits_{#1}}}
                  
                  \newcommand{\aveint}[2]{\mathchoice%
          {\mathop{\kern 0.2em\vrule width 0.6em height 0.69678ex depth -0.58065ex
                  \kern -0.8em \intop}\nolimits_{\kern -0.45em#1}^{#2}}%
          {\mathop{\kern 0.1em\vrule width 0.5em height 0.69678ex depth -0.60387ex
                  \kern -0.6em \intop}\nolimits_{#1}^{#2}}%
          {\mathop{\kern 0.1em\vrule width 0.5em height 0.69678ex depth -0.60387ex
                  \kern -0.6em \intop}\nolimits_{#1}^{#2}}%
          {\mathop{\kern 0.1em\vrule width 0.5em height 0.69678ex depth -0.60387ex
                  \kern -0.6em \intop}\nolimits_{#1}^{#2}}}

\begin{document}

\title{\bf A local/nonlocal diffusion model}

\author{Bruna C. dos Santos, Sergio M. Oliva and Julio D. Rossi}

\maketitle

\begin{abstract} 
In this paper, we study some qualitative properties for an evolution problem that combines local and nonlocal diffusion operators acting in two different subdomains and, coupled in such a way that, the resulting evolution problem is the gradient flow of an energy functional. 
The coupling takes place at the interface between the regions in which the different diffusions take place. 
We prove existence and uniqueness results, as well as, that the model preserves the total mass of the initial condition. 
We also study the asymptotic behavior of the solutions. Finally, we show a suitable way to recover the heat equation at the whole domain from taking the limit at the nonlocal rescaled kernel.
\end{abstract}

\noindent{\makebox[1in]\hrulefill}\newline2010 \textit{Mathematics Subject
Classification.} 
35K55, 
35B40, 
35A05. 


\section{Introduction and main results}

In this paper we combine a local diffusion equation, the classical heat equation,
\begin{equation} \label{ec.calor}
\frac{\partial u}{\partial t} (x,t) = \frac{\partial^2 u}{\partial x^2} (x,t)
\end{equation}
with a nonlocal diffusion equation with an integrable kernel
\begin{equation} \label{ec.no-local}
\frac{\partial u}{\partial t} (x,t) = \int J(x-y) (u(y,t) -u(x,t) ) dy.
\end{equation}
The kernel $J(z)$ is assumed to be nonnegative, continuous, symmetric, compactly supported with
$supp (J) =[-R,R]$ and
$\int J(x,y)\, dy  =1$ (these hypotheses on $J$ will be assumed from now on)

We aim to obtain a model that couples the local heat equation \eqref{ec.calor} with 
the nonlocal problem \eqref{ec.no-local} in such a way that the following features (that are the usual ones when one 
deals with a diffusion problem) hold:
\begin{itemize}
\item The problem is well-posed in the sense that there are existence and uniqueness of solutions. Besides, a comparison 
principle holds.

\item There is an energy functional such that the evolution problem can be view 
as the gradient flow associated with this energy.

\item The total mass of the initial condition is preserved along with the evolution.

\item Solutions converge exponentially fast to the mean value of the initial condition. 
\end{itemize}

Let us describe our model in terms of a particle system. To simplify the exposition we 
will restrict ourselves to a one-dimensional problem and comment on the
extension to higher dimensions at the end of the paper.
We split the domain $\Omega = (-1,1)$ into two subdomains $(-1,0)$ and $(0,1)$
 (to simplify 
we will restrict ourselves to this simple configuration). In $(-1,0)$ particles move by Brownian
motion (this gives the equation $\frac{\partial u}{\partial t} (x,t) = \frac{\partial^2 u}{\partial x^2} (x,t)$, $x\in (-1,0)$) with a reflexion at $x=-1$
(then $\frac{\partial u}{\partial x} (-1,0) =0$) and when the particle arrives to $x=0$ it passes trough to the other subdomain, $(0,1)$
(this will give a flux boundary condition at $x=0$). On the other hand, in $(0,1)$ particles obey a pure jump process
with jumping probability given by $J(x-y)$ (this gives an equation of the form \eqref{ec.no-local} in $(0,1)$, when a 
particle that is at $x\in (0,1)$ wants to jump to a location $y\in (-1,0)$ it enters the domain $(-1,0)$ at the point
$x=0$ (particles are stuck there, giving the counterpart to the flux coming  from $(-1,0)$).
This process has a density $w(x,t)$, which obeys an evolution equation associated with the gradient flow
of a local/nonlocal energy that we describe in the next section.

For different couplings between local and nonlocal models, we refer to \cite{delia2,delia3,Du,Gal,GQR,Kri} and references therein.
In \cite{delia2}, local and nonlocal problems are coupled trough a prescribed region in which
both kinds of equations overlap (the value of the solution in the nonlocal part of the domain is used as a Dirichlet boundary condition for the local part and vice-versa). This kind of coupling gives continuity of the solution in the overlapping region but does not preserve the total mass.
In \cite{delia2} and \cite{Du}, numerical schemes using local and nonlocal equations were developed and used to improve the computational
accuracy when approximating a purely nonlocal problem. In \cite{GQR} (see also \cite{Gal,Kri}), an energy closely related to ours was studied, but the gradient flow
of this energy (that it has all the nice properties listed above) gives an equation in the local region in which the coupling with
the nonlocal part appears as an external source in the heat equation (that is complemented with zero flux boundary conditions in the whole
boundary of the local region). In 
probabilistic terms, in the model described in  \cite{GQR}, particles may jump across the interface between the two regions but can not pass coming from the local side unless they jump.

\subsection{A local/nonlocal diffusion model}
As we mentioned, let us consider as the reference domain $\Omega=(-1,1) \subset \mathbb{R}$ that is divided in two
disjoint regions, the intervals $\Omega_{l}=(-1,0)$ and $\Omega_{nl}=(0,1)$, the local and nonlocal domains, respectively. 
We split a function $w\in L^2 (-1,1)$ as $w =u+v$, with $u =w \chi_{(-1,0)}$
and $v = w \chi_{(0,1)}$. For any $$w=(u,v) \in \mathcal{B}:=\left\{w \in L^2(-1,1): u \in H^1(-1,0),  v \in L^2(0,1) \right\}$$ we define the energy
\begin{equation} \label{energy}
\begin{array}{l}
    E(u,v)  \displaystyle:=\frac{1}{2}\int_{-1}^{0} \Big|\frac{\partial u}{\partial x}\Big|^2 dx + \frac{C_{J,1}}{4}\int_{0}^{1}\int_{0}^{1}J(x-y)\left(v(y)-v(x)\right)^2 dydx \\[10pt]
    \displaystyle \quad \qquad \qquad + \frac{C_{J,2}}{2}\int_{0}^{1}\int_{-1}^{0}J(x-y)dy\left(v(x)-u(0)\right)^2 dy dx,
\end{array}
\end{equation}
where $C_{J,1}$ and $C_{J,2}$ are fixed positive constants. Notice that in this energy functional we have two terms
$$
\frac{1}{2}\int_{-1}^{0}\Big|\frac{\partial u}{\partial x}\Big|^2 dx \qquad \mbox{and} \qquad \frac{C_{J,1}}{4}\int_{0}^{1}\int_{0}^{1}J(x-y)\left(v(y)-v(x)\right)^2 dydx
$$
that are naturally associated with the equations \eqref{ec.calor} and \eqref{ec.no-local}, plus a coupling term 
$$
\frac{C_{J,2}}{2}\int_{0}^{1}\int_{-1}^{0}J(x-y)dy\left(v(x)-u(0)\right)^2 dy dx
$$
that involves only the value of $u$ only at $x=0$.

Our aim is to follow the gradient flow associated with this energy, that is,
$(u,v)$ will be the solution of the the abstract $ODE$ problem $$(u,v)'(t)=-\partial E\left[(u,v)(t)\right], \qquad t\geq 0,$$ 
with $u(0)=u_{0}$, $v(0) =v_0$and, where $\partial E\left[(u,v)\right]$ denotes the subdifferential of $E$ at the point $(u,v)$. 
Let us compute the derivative of $E$ at $(u,v)$, in the direction of $\varphi \in C_{0}^{\infty}(-1,1)$ that is given by
\begin{align*}
 \displaystyle 
 \partial_{\varphi}E(u,v)  & \displaystyle =\lim_{h \to 0}\frac{E(u+h\varphi,v+h\varphi)-E(u,v)}{h} \\[10pt]
 \displaystyle  & = \int_{0}^{-1} \frac{\partial u}{\partial x} \frac{\partial \varphi}{\partial x} dx + \frac{C_{J,1}}{2}\int_{0}^{1}\int_{0}^{1}J(x-y)(v(y)-v(x))(\varphi(y)-\varphi(x))dydx  \\[10pt]
 \displaystyle  & \qquad + \frac{C_{J,2}}{2}\int_{0}^{1}\int_{-1}^{0}J(x-y)(v(x)-u(0))(\varphi(x)-\varphi(0))dydx.
\end{align*}
Thus, if $u$ is smooth, we would have
\begin{align*}
 \displaystyle 
& \partial_{\varphi}E(u,v) = \left\{\frac{\partial u}{\partial x}(0)- C_{J,2} \int_{0}^{1}\int_{-1}^{0}J(x-y)(v(x)-u(0))dydx \right\}\varphi(0) - \frac{\partial u}{\partial x}
(-1)\varphi(-1) \\[10pt]
\displaystyle &  \qquad\qquad\qquad - \int_{-1}^{0} \frac{\partial^2 u}{\partial x^2} \varphi dx - C_{J,1} \int_{0}^{1} \left\{ \int_{0}^{1}J(x-y)(v(y)-v(x))dy \right\} \varphi(x) dx 
\\[10pt]
\displaystyle &  \qquad\qquad\qquad + C_{J,2}\left\{ \int_{-1}^{0}J(x-y)(v(x)-u(0)) dy \right\} \varphi(x) dx . 
\end{align*}

Since $\langle \partial E[u,v],\varphi \rangle=\partial_{\varphi}E(u,v)$, we can derive the local/nonlocal problem associated to this gradient flow. The evolution problem consists of two parts. A local part, composed of a heat equation with Neumann/Robin type boundary conditions,
\begin{align}\label{local}
\begin{cases}
 \displaystyle   \frac{\partial u}{\partial t}(x,t)  =  \frac{\partial^{2}u }{\partial x^{2}}(x,t),  \\[10pt]
 \displaystyle   \frac{\partial u}{\partial x}(-1,t)  =  0,  \\[10pt]
 \displaystyle    \frac{\partial u}{\partial x}(0,t)  = C_{J,2}\int_{-1}^{0}\int_{0}^{1}J(x-y)( v(y,t)-  u(0,t))dydx,  \\[10pt] 
 \displaystyle    u(x,0)  =u_{0}(x),
    \end{cases}
    \end{align}
    for $ x \in (-1,0)$, $t>0$. Notice that we have a Robin type boundary condition at $x=0$ that encodes the coupling with the nonlocal part of the 
    problem.

For the nonlocal domain we have,
\begin{align}\label{nonlocal}
\begin{cases}
 \displaystyle \frac{\partial v}{\partial t}(x,t)  =C_{J,1}\int_{0}^{1} J(x-y)\left(v(y,t)-v(x,t) \right)dy 
   - C_{J,2}\int_{-1}^{0}J(x-y)  dy  ( v(x,t)-  u(0,t)),  \\[10pt]
     v(x,0)  =v_{0}(x), 
      \end{cases}
     \end{align} 
     for $ x \in (0,1)$, $t>0$. 
    Here we have a nonlocal diffusion problem for $v$, where the coupling with the local part $u$ appears as a source term in the equation, while the value of $u$ appears only at the interface $x=0$.
 
 The complete problem can be summarized as follows: we look for $w$ defined by
\begin{align}\label{complete}
w(x,t) = 
\begin{cases}
   u(x,t), \quad \text{if} \quad x \in (-1,0), \\
    v(x,t), \quad \text{if} \quad x \in (0,1), 
    \end{cases}
    \end{align}
where $(u,v)$ is a solution to \eqref{local}--\eqref{nonlocal}.

For this problem we have the following result:
\begin{theorem}
Given $w_{0} = (u_0,v_0) \in L^2(-1,1)$, there exists an unique mild solution 
$$w(\cdot,t) \in \mathcal{B}:=\left\{w \in L^2(-1,1): u \in H^1(-1,0),  v \in L^2(0,1) \right\}$$ 
to the local/nonlocal problem \eqref{complete} with $(u,v)$ satisfying \eqref{local}--\eqref{nonlocal} that is globally defined. 
If $w_{0} = (u_0,v_0)$ with $u_0 \in C([-1,0])$ and $v_0 \in C([0,1])$ then the solution $(u,v)$ is such that
$u(\cdot,t) \in C([-1,0])$ and $v (\cdot,t) \in C([0,1])$ for every $t>0$.

A comparison principle holds: if the initial data are ordered, $w_{0} \geq z_{0}$, then the corresponding solutions 
are also ordered, they verify $w \geq z$ in $(-1,1)\times \mathbb{R_{+}}$. 

Moreover, the total mass of the solution is preserved along the evolution, that is, 
$$\int_{-1}^1w(x,t) dx =\int_{-1}^1 w_{0}(x) dx = \int_{-1}^0 u_{0}(x) dx+ \int_{0}^1 v_{0}(x) dx.$$
\end{theorem}

\begin{remark}{\rm 
Notice that we prove that for a continuous initial condition we obtain a solution
 $(u,v)$ such that
$u(\cdot,t) \in C([-1,0])$ and $v (\cdot,t) \in C([0,1])$ for every $t>0$, but we are not
imposing (nor obtaining) continuity across the interface, that is, we don't necessarily 
have $u(0,t)= v(0,t)$.}
\end{remark}

\subsection{Asymptotic behavior}
Once we proved the existence and uniqueness of a global solution, our next goal is to look for its asymptotic behavior as $t \to \infty$. 
We start by observing that the constants, $w(x,t) \equiv cte$, are stationary solutions of \eqref{local}--\eqref{nonlocal}.

For the heat equation
\begin{equation}
  \frac{\partial u}{\partial t} = \frac{\partial^2 u}{\partial x^2},  
\end{equation} 
with Neumann boundary conditions, it is well known that solutions
have an exponential time decay to the mean value of the initial condition, that is,
$$
\left\| u(\cdot,t) - \fint u_0  \right\|_{L^2} \leq C (u_0) e^{-\beta t}.
$$
The same is valid (with a different $\beta$) for solutions to the nonlocal heat equation 
\begin{equation}
     \frac{\partial v}{\partial t} (x,t)=\int_{0}^1 J(x-y)(v(y,t)-v(x,t))dy,
\end{equation}
with the additional assumption on the kernel,
$
    M(J):=\int_{\mathbb{R}}J(z)|z|^{2}dz<\infty,
$
see \cite{andreu2010nonlocal,ChChRo}.

Here, we will show that solutions to our problem \eqref{complete} have the same behavior, that is, 
the solution of the coupled local-nonlocal problem converges exponentially to mean value of the initial condition.

\begin{theorem} \label{teo.asymp.intro}
Given $w_{0} \in L^2(-1,1)$, the solution to \eqref{complete} with initial condition $w_0$ 
converges to its mean value as $t \to \infty$ with an exponential rate.
\begin{equation}
    \left\| w(\cdot ,t) - \fint w_0 \right\|_{L^2(-1,1)} \leq C e^{-\beta_{1} t}, \qquad t>0,
\end{equation}
where $\beta_{1}>0$ depends only on $J$ and $\Omega$ and, the constant $C$ depends on the initial condition $w_{0}$.

\end{theorem}

\subsection{Rescaling the kernel}
In the following, we will show that the solutions of the evolution problem \eqref{local}-\eqref{nonlocal}, with the kernel $J$ rescaled suitably, converge to the classical local problem given by the heat equation at the whole domain. Consider the rescaled kernel given by
\begin{equation} \label{rescale-J}
   J^{\varepsilon}(x):= \frac{1}{\varepsilon^{3}}J\left(\frac{x}{\varepsilon}\right).  
\end{equation}
From now, we choose (and fix) the constants $C_{J,1}$ and $C_{J,2}$ that appears before the nonlocal terms as
 \begin{equation}
  C_{J,1}:= \frac{2}{M(J)},
\qquad
 \mbox{ and } \qquad
  C_{J,2}:= 1. 
 \end{equation}

In fact, our goal now is to show that the solutions of the local heat equation with Neumann boundary conditions,
\begin{align}\label{localcomplete}
\begin{cases}
   \displaystyle  \frac{\partial w}{\partial t}(x,t)=  \frac{\partial^2 w}{\partial x^2}(x,t), \quad x \in (-1,1), \quad t>0,\\[10pt]
    \displaystyle  \frac{\partial w}{\partial x}(-1,t)=  \frac{\partial w}{\partial x} (1,t)=0, \quad t>0,\\[10pt]
    w(x,0)=w_{0}(x), \quad x \in (-1,1),
    \end{cases}
\end{align}
can be obtained as the limit as $\varepsilon \to 0$ of the solution $w^{\varepsilon}$ to our local/nonlocal problem
with $J$ replaced by $J^\varepsilon$, given by \eqref{rescale-J}.
We will call $w^\varepsilon = (u^\varepsilon, v^\varepsilon)$ the solution to \eqref{local}--\eqref{nonlocal} with the rescaled kernel and
a fixed initial condition $w(x,0)=u_{0}(x)\chi_{(-1,0)} (x)+v_{0}(x) \chi_{(0 ,1)} (x)$ (here $u_{0}(x)$, $v_{0}(x)$ are fixed).

We have the following result:
\begin{theorem}
Let $w_{0}$ $\in$ $L^{2}(-1,1)$. For each $\varepsilon >0$, let $w^{\varepsilon}$ be the solution to \eqref{local}-\eqref{nonlocal} 
with $J$ replaced by $J^\varepsilon$ given by \eqref{rescale-J} and initial condition $w_{0}$. Then, it holds that
\begin{equation}
    \lim_{\varepsilon \to 0}\left(\max_{t \in [0,T]}\parallel w^{\varepsilon}(\cdot,t)-w(\cdot,t) \parallel_{L^2(-1,1)}\right)=0
 \end{equation}
where $w$ is the solution to \eqref{localcomplete}.
\end{theorem}

ORGANIZATION OF THE PAPER: The paper is organized as follows: In Section 2, 
we prove a key result concerning the control of the pure nonlocal energy by our local/nonlocal energy. 
In Section 3, we prove the existence and uniqueness of the problem, the total mass conservation property and the asymptotic behavior of the solutions for large times. In Section 4, we deal with the rescaling of the kernel.
Finally, in the final section (Section \ref{sect-higher}), we explain how to extend our results to higher dimensions.

\section{Preliminaries}

\subsection{Control of the nonlocal energy}
In this section, we prove the first crucial lemma that ensures domination of our energy over the pure nonlocal energy.

\begin{lemma}\label{energy.lema}
Let $$(u,v) \in \mathcal{B}:=\left\{u \in H^1(-1,0),  v \in L^2(0,1) \right\}.$$ 
Then, there exists a constant $k:=k(J,\Omega)>0$ such that
\begin{equation} \label{energycontrol}
\begin{array}{l}
\displaystyle
 \frac{1}{2} \int_{-1}^{0} \Big| \frac{\partial u}{\partial x}\Big|^2 dx + \frac{C_{J,1}}{4}\int_{0}^{1}\int_{0}^{1}J(x-y)\left(v(y)-v(x)\right)^2 dydx 
  + \frac{C_{J,2}}{2}\int_{0}^{1}\int_{-1}^{0}J(x-y)\left(v(x)-u(0)\right)^2 dydx  \\[10pt]
  \displaystyle \qquad  \geq k \int_{-1}^{1}\int_{-1}^{1}J(x-y)\left(w(y)-w(x)\right)^2 dydx.
\end{array}
\end{equation}
\end{lemma}

\begin{proof}
Assume that the conclusion does not hold. Then, there exists a sequence $\{w_{n}\} \in L^2(-1,1)$, $\{u_{n}\} \in H^1(-1,0)$ and $\{v_{n}\} \in L^2(0,1)$, 
such that 
\begin{equation} \label{jjj}
    \int_{-1}^1 \int_{-1}^1 J(x-y) (w_n(y) - w_n(x))^2 dy dx =1,
\end{equation}
and satisfying,
\begin{equation}\label{mass}
 \int_{-1}^{1} w_{n}= \int_{-1}^{0} u_{n} + \int_{0}^{1} v_{n} = 0,
\end{equation}
and
\begin{align}
  1 & \geq n \left( \frac{1}{2}\int_{-1}^0 \Big| \frac{\partial u_n}{\partial x} \Big|^{2} +\frac{C_{J,1}}{4} \int_{0}^{1} \int_{0}^{1} J(x-y) (v_n(y) - v_n(x))^2 dy dx + \frac{C_{J,2}}{2}\int_{-1}^{0} \int_{0}^{1} J(x-y) (u_n(0) - v_n(x))^2 dx dy\right),
  \end{align}
for every  $n \in \mathbb{N}$.

Consequently, taking the limit in $n$, we obtain
\begin{equation}\label{lim1}
    \lim_{n}\left(\frac{1}{2}\int_{-1}^0  \Big| \frac{\partial u_n}{\partial x} \Big|^{2} \right)=0,
\end{equation}
\begin{equation}\label{lim2}
    \lim_{n}\left(\frac{C_{J,1}}{4}\int_{0}^{1} \int_{0}^{1} J(x-y) (v_n(y) - v_n(x))^2 dy dx\right)=0,
\end{equation}
and
\begin{equation}\label{lim3}
    \lim_{n}\left(\frac{C_{J,2}}{2}\int_{-1}^{0} \int_{0}^{1} J(x-y) (u_n(0) - v_n(x))^2 dx dy\right)=0.
\end{equation}

 From \eqref{lim1} together with \eqref{jjj} that implies a bound on the $L^2$-norm of $w_n$, we can take a subsequence, also denoted $\{u_{n}\}$, which weakly converge for some limit in $H^1(-1,0)$ that is given by a constant $A$,  
\begin{align}\nonumber
&  u_n \to A  \quad \text{in} \quad  L^{2}(-1,0) \quad \text{and}, \\
&  u_n \to A \quad \text{uniformly} \quad \text{in} \quad (-1,0).
\end{align}
Note that, in particular, $u_{n}(0) \to A$. Besides that, from equation \eqref{lim2} we also can take a subsequence, also denoted as $\{v_{n}\}$, which strongly converge for some limit in $L^2(0,1)$ that is given by a constant $B$.

From the \eqref{lim3}, we obtain $A=B$. 
Moreover, from equation \eqref{mass} we get that $A+B=0$. 
Therefore, we get
$
A=B=0$.
On the other hand, we have
\begin{align}\nonumber
 & \int_{-1}^1 \int_{-1}^1 J(x-y) (w_n(y) - w_n(x))^2 dy dx =1, 
\end{align}
which implies
\begin{equation*}
     \int_{-1}^1 \int_{-1}^1 J(x-y) (A - B)^2 dy dx =1, 
\end{equation*}
that contradicts $A=B=0$.
\end{proof}

The main advantage of this estimate is to observe that the constant obtained from \eqref{energycontrol} 
can be taken independent of $\varepsilon$ when we consider the rescaled kernel $J^\varepsilon$. In fact,
since a simple inspection of the previous proof gives that we do not have any dependence on $\varepsilon$ in the constant $k$.

\subsection{A Poincaré type inequality}
Let us consider $w_{\varepsilon}$ as in the introduction, that is,
\begin{align*}
w_{\varepsilon} (x)=
\begin{cases}
u_{\varepsilon}(x), \quad \text{if} \quad x \in (-1,0)\\
u_{\varepsilon}(x),\quad \text{if} \quad x \in (0,1).
\end{cases}
\end{align*} 

From \cite{andreu2010nonlocal} we have that
\begin{lemma}
There exists a constant $C>0$ (independent of $\epsilon$) such that, for every $\{w_{\varepsilon_{n}}\} \in L^2(-1,1)$ it holds 
\begin{equation}\label{poincareine}
\begin{array}{l}
\displaystyle 
\int_{-1}^{1}\left|w_{\varepsilon_{n}}(x)-\fint_{-1}^{1}w_{\varepsilon_{n}}(x)dx \right|^2 dx  
 \leq C \frac{1}{\varepsilon_{n}^3}\int_{-1}^1 \int_{-1}^1 J\left(\frac{x-y}{\varepsilon_{n}}\right) (w_{\varepsilon_n}(y) - w_{\varepsilon_n}(x))^2 dy dx.
\end{array}
\end{equation}
\end{lemma}

As consequence of \eqref{poincareine} and the control of the nonlocal energy given by
\eqref{energycontrol}
 we have the following Poincaré type inequality.

\begin{lemma}
Let $w_{\varepsilon} \in \mathcal{B}:=\left\{u_{\varepsilon} \in H^1({(-1,0)}),  v_{\varepsilon} \in L^2((0,1)) \right\}$. 
Then there exists a constant $k:=k(J,\Omega)>0$, independent of $\varepsilon$, such that
\begin{equation}\label{controlenergy2}
\begin{array}{l}
\displaystyle
 \frac{1}{2} \int_{-1}^{0}\Big| \frac{\partial w_\varepsilon}{\partial x} \Big|^{2} dx + \frac{C_{J,1}}{4\varepsilon^{3}}\int_{0}^{1}\int_{0}^{1}\left(\frac{J(x-y)}{\varepsilon}\right)\left(w_{\varepsilon}(y)-w_{\varepsilon}(x)\right)^2 dydx \\[10pt]
\displaystyle \qquad
  + \frac{C_{J,2}}{2\varepsilon^{3}}\int_{0}^{1}\int_{-1}^{0}J\left(\frac{x-y}{\varepsilon}\right)\left(w_{\varepsilon}(x)-w_{\varepsilon}(0)\right)^2 dydx  \\[10pt]
  \displaystyle \geq k \frac{1}{\varepsilon^{3}}\int_{-1}^{1}\int_{-1}^{1}J\left(\frac{x-y}{\varepsilon}\right)\left(w_{\varepsilon}(y)-w_{\varepsilon}(x)\right)^2 dydx.
\end{array}
\end{equation}
\end{lemma}
\begin{proof}
Let us argue by contradiction. Suppose that \eqref{controlenergy2} is false. Then, for every $n \in \mathbb{N}$, there exists a subsequence $\varepsilon_{n} \to 0$, and $\{w_{\varepsilon_n}\} \in L^2(-1,1)\cap H^1(-1,0)$, such that 
\begin{equation}\label{Ia}
 \int_{-1}^{1} w_{\varepsilon_n}= \int_{-1}^{0} u_{\varepsilon_n} + \int_{0}^{1} v_{\varepsilon_n} = 0,
\end{equation}
\begin{equation}\label{IIa}
    \frac{1}{\varepsilon_{n}^3}\int_{-1}^1 \int_{-1}^1 J\left(\frac{x-y}{\varepsilon_{n}}\right) (w_{\varepsilon_n}(y) - w_{\varepsilon_n}(x))^2 dy dx =1,
\end{equation}
and,
\begin{align}\label{IIIa}
  \frac1{n} & \geq \left( \frac{1}{2}\int_{-1}^0 \Big| \frac{\partial w_{\varepsilon_n}}{\partial x} \Big|^{2} +\frac{C_{J,1}}{4\varepsilon_{n}^3} \int_{0}^{1} \int_{0}^{1} J\left(\frac{x-y}{\varepsilon_{n}}\right) (w_{\varepsilon_n}(y) - w_{\varepsilon_n}(x))^2 dy dx \right. \\ \nonumber
& \qquad \qquad  \qquad  \qquad  \qquad  + \left. \frac{C_{J,2}}{2\varepsilon_{n}^3}\int_{-1}^{0} \int_{0}^{1} J\left(\frac{x-y}{\varepsilon_{n}}\right) (w_{\varepsilon_{n}}(0) - w_{\varepsilon_n}(x))^2 dx dy\right), \quad \forall  n \in \mathbb{N.}
\end{align}

Taking the limit in $n$ in \eqref{IIIa}, we obtain
\begin{equation}\label{lim11}
    \lim_{n}\left(\frac{1}{2}\int_{-1}^0 \Big| \frac{\partial w_{\varepsilon_n} }{\partial x} \Big|^{2}\right)=0,
\end{equation}
\begin{equation}\label{lim21}
    \lim_{n}\left(\frac{C_{J,1}}{4\varepsilon_{n}^3} \int_{0}^{1} \int_{0}^{1} J\left(\frac{x-y}{\varepsilon_{n}}\right) (w_{\varepsilon_n}(y) - w_{\varepsilon_n}(x))^2 dy dx\right)=0,
\end{equation}
and
\begin{equation}\label{lim31}
    \lim_{n}\left(\frac{C_{J,2}}{2\varepsilon_{n}^3}\int_{-1}^{0} \int_{0}^{1} J\left(\frac{x-y}{\varepsilon_{n}}\right) (w_{\varepsilon_{n}}(0) - w_{\varepsilon_n}(x))^2 dx dy\right)=0.
\end{equation}

 From \eqref{lim11} we have that $w_{\varepsilon_{n}}$ is bounded in $H^1(-1,0)$, so passing to a subsequence, also denoted $\{w_{\varepsilon_{n}}\}$, such that $\varepsilon_{n} \to 0$, we have
\begin{align}\nonumber
&  w_{\varepsilon_{n}} \rightharpoonup w  \quad \text{in} \quad  H^1(-1,0), \quad \\\nonumber
&  w_{\varepsilon_{n}} \to w \quad \text{in} \quad  L^2(-1,0) \quad \text{and} \\\nonumber
&  w_{\varepsilon_{n}} \quad \text{converges} \quad \text{uniformly} \quad \text{in} \quad (-1,0).
\end{align}
Thanks to \eqref{lim11} and Fatou's Lemma we also know that
\begin{equation}\label{1a}
\frac{1}{2}\int_{-1}^{0}\Big| \frac{\partial w}{\partial x} \Big|^{2} \leq 
\liminf_{\varepsilon_{n} \to 0}\frac{1}{2}\int_{-1}^{0}\Big| \frac{\partial w_{\varepsilon_n}}{\partial x} \Big|^{2} =0.
\end{equation}
Hence, the limit is a constant, let us call $w=A_{1} \in H^1(-1,0)$. 

Now, we shall see that $\{w_{\varepsilon_{n}}\}$ is also bounded in $L^2(0,1)$ to see that, as $\varepsilon_{n} \to 0$, $\{w_{\varepsilon_{n}}\}$ weakly converges in $L^2(0,1)$ to some limit $w$ (that will be a constant $A_{2}$).

Thanks to \eqref{poincareine} we have
\begin{align*}
\int_{-1}^{1} \left|w_{\varepsilon_{n}}(x)-\fint_{-1}^{1}w_{\varepsilon_{n}}(x)dx \right|^2 dx \leq C,
\end{align*} 
which implies 
\begin{align}\label{boundedvn}
\int_{-1}^{1}|w_{\varepsilon_{n}}(x)dx |^2 dx =\int_{-1}^{0}|u_{\varepsilon_{n}}(x)dx |^2 dx+\int_{0}^{1}|v_{\varepsilon_{n}}(x)dx |^2 dx \leq C.
\end{align} 
Thanks to \eqref{boundedvn} we have that $v_{\varepsilon_{n}}$ is bounded in $L^2(0,1)$ and then, there exists a subbsequence, also denoted by $v_{\varepsilon_{n}}$ which weakly converges for some limit $w \in L^2(0,1)$.

Now, from \eqref{lim21}, changing variables $x=y+\varepsilon_{n}z$ we obtain
 \begin{equation}\label{variablechange99}
 \begin{array}{l}
 \displaystyle 
\frac{C_{J,1}}{4\varepsilon_{n}^3} \int_{0}^{1} \int_{0}^{1} J\left(\frac{x-y}{\varepsilon_{n}}\right) (w_{\varepsilon_n}(y) - w_{\varepsilon_n}(x))^2 dy dx
 =\frac{C_{J,1}}{4} \int_{0}^{1} \int_{\frac{-y}{\varepsilon_n}}^{\frac{1-y}{\varepsilon_n}} J(z) \frac{(w_{\varepsilon_n}(y) - w_{\varepsilon_n}(y+\varepsilon_{n}z))^2}{\varepsilon_{n}^2} dz dy.
\end{array}
\end{equation}  
As the limit in \eqref{lim21} is zero, it follows that
\begin{equation}\label{boundeddifference}
\frac{C_{J,1}}{4} \int_{0}^{1} \int_{\frac{-y}{\varepsilon_n}}^{\frac{1-y}{\varepsilon_n}} J(z) \frac{(w_{\varepsilon_n}(y) - w_{\varepsilon_n}(y+\varepsilon_{n}z))^2}{\varepsilon_{n}^2} dz dy \leq C.
\end{equation}

So, thanks to \eqref{boundeddifference} and the weak convergence of $\{v_{\varepsilon_{n}}\}$ to $w$ in $L^2(0,1)$, by [\cite{andreu2010nonlocal}, Theorem 6.11] we have that $w \in H^1(0,1)$ and, moreover
\begin{equation}
\left(\frac{C_{J,1}}{4} J(z)\right)^{1/2} \frac{(w_{\varepsilon_n}(y) - w_{\varepsilon_n}(y+\varepsilon_{n}z))}{\varepsilon_{n}}  \rightharpoonup \left(\frac{C_{J,1}}{4} J(z)\right)^{1/2} z \cdot \frac{\partial w}{\partial {x}} (y)
\end{equation}
weakly in $L^2(0,1)\times L^2(\mathbb{R})$. Therefore taking the limit $\varepsilon_{n} \to 0$ in \eqref{variablechange99} we get, 
\begin{equation}\label{2a}
\frac{1}{2}\int_{0}^{1} \Big| \frac{\partial w}{\partial {x}}\Big|^2 =0.
\end{equation}
Hence $w=A_2$ is just a constant.

Finally, from \eqref{lim31}, taking $\varepsilon_{n} \to 0$ and by the Monotone Convergence Theorem, we obtain that $A_{1}=A_{2}$. Moreover, from equation \eqref{Ia} we get that $A_{1}+A_{2}=0$ which contradicts \eqref{IIa}.
\end{proof}

\section{The local/nonlocal problem} \label{sect-proofs}
\subsection{Existence and uniqueness}
Now, our goal is to show the existence and uniqueness of solutions. The main idea to prove this result is, given a function $u$ defined for $x \in [-1,0]$ we will use it as an initial input for the equation \eqref{nonlocal} in $[0,1]$. The solution $v$ of this problem is then used to solve the equation \eqref{local} in $[-1,0]$, which yields a function $z$. This procedure in two steps can be regarded as an operator $H$ given by $H(u)=z$. Now our task is
to look for a fixed point of $H$ via contraction in an adequate norm, meaning that, there must exists $u=H(u)$, solving the equation for $x \in [-1,0]$ with its corresponding $v$ solving the equation for $x \in [0,1]$.

Fix $T>0$ and consider the Banach spaces
\begin{equation*}
X_{T}= \left\{u \in C([-1,0]\times [0,T])\right\} \quad \text{and} \quad Y_{T}= \left\{v \in C([0,1]\times [0,T])\right\},
\end{equation*}
with the respective norms
\begin{equation*}
\parallel u \parallel_{l} = \max_{t \in [0,T]} \max_{x \in [-1,0]} |u| \quad \text{and} \quad \parallel v \parallel_{nl} = \max_{t \in [0,T]} \max_{x \in [0,1]} |v| .
\end{equation*}
Given $T>0$, we define the operator $H_{1}:X_{T} \to Y_{T}$ as $H_{1}(u)=v$, where $v$ is the unique solution of 
\begin{align}\label{FPnonlocal}
\begin{cases}
 \displaystyle \frac{\partial v}{\partial t}(x,t) = C_{J,1}\int_{0}^{1} J(x-y)\left(v(y,t)-v(x,t) \right)dy  - C_{J,2}\int_{-1}^{0}J(x-y)dy \, v(x,t) 
  +  C_{J,2}u(t,0)\int_{-1}^{0}J(x-y)dy,  \\[10pt]
     v(0,x)  =v_{0}(x), 
      \end{cases}
     \end{align} 
for $x \in (0,1)$ and $t \in (0,T)$.

In the next lemma we will show that this problem has an unique solution (that means that $H_1$ is well defined). 
In addition, we show continuous dependence on $u$.

\begin{lemma}
There are constants $C_{J,i}$, $i=1,2$, depending only on $J$, such that for $T$ $\in$ $\left(0,\frac{1}{2C_{J,1}+C_{J,2}}\right)$,
given $u(x,t)$ $\in$ $C([-1,0]\times [0,T])$ and $v_{0}$ $\in C([0,1])$, there exists an unique 
$v(x,t) \in C([0,1]\times [0,T])$, solution to \eqref{nonlocal}. Moreover, if $v_{1}$ and $v_{2}$ are the solutions corresponding to $u_{1}$ and $u_{2}$ then 
\begin{equation}\label{datadependence1}
\parallel v_{1}-v_{2}\parallel_{nl} \leq \frac{C_{J,2}T}{1-(2 C_{J,1}+C_{J,2})T}\parallel u_{1}-u_{2}\parallel_{l}.
\end{equation}
\end{lemma}
\begin{proof}
To show the existence and uniqueness we will use a fixed point argument. Let us define an operator $A_{u}(v): Y_{T} \to Y_{T}$ as
\begin{align*}
& A_{u}(v)(t,x)  :=v_{0}(x)+C_{J,1}\int_{0}^{t}\int_{0}^{1}J(x-y)(v(y,s)-v(x,s))dyds \\[10pt]
& \qquad\qquad\qquad\qquad - C_{J,2}\int_{0}^{t}\int_{-1}^{0}J(x-y)v(x,s)dy ds+C_{J,2}\int_{0}^{t}\int_{-1}^{0}J(x-y)u(0,s)dyds.
\end{align*}
Taking the difference $A_{u}(v_{1})-A_{u}(v_{2})$ we get
\begin{align*}
\displaystyle
\parallel A_{u}(v_{1})-A_{u}(v_{2})\parallel_{nl} &\leq C_{J,1} \max_{t \in [0,T]}\max_{x \in [0,1]} \int_{0}^{t}\int_{0}^{1}J(x-y)|v_{1}(y,s)-v_{2}(y,s)|dyds \\[10pt]
& \qquad + C_{J,1} \max_{t \in [0,T]}\max_{x \in [0,1]} \int_{0}^{t}\int_{0}^{1}J(x-y)|v_{2}(x,s)-v_{1}(x,s)|dyds \\[10pt]
& \qquad + C_{J,2} \max_{t \in [0,T]}\max_{x \in [0,1]} \int_{0}^{t}\int_{-1}^{0}J(x-y)|v_{2}(x,s)-v_{1}(x,s)|dyds.
\end{align*}
Since $J \geq 0$ and $\int_{\mathbb{R}}J=1$, applying Fubini's theorem, we obtain
\begin{equation*}
\parallel A_{u}(v_{1})-A_{u}(v_{2})\parallel_{nl} \leq (2 C_{J,1}+C_{J,2})T\parallel v_{1}-v_{2}\parallel_{nl}.
\end{equation*}
Choosing $T<\frac{1}{2C_{J,1}+C_{J,2}}$, $A_{u}$ is a strict contraction, and hence it has an unique fix point.

To check the dependence on the data, since $v_{1}=A_{u_{1}}(v_{1})$ and $v_{2}=A_{u_{2}}(v_{2})$, if we follow the same idea, we will get
\begin{equation*}
\parallel v_{1}-v_{2} \parallel_{nl} \leq (2 C_{J,1}+C_{J,2})T\parallel v_{1}-v_{2}\parallel_{nl}+C_{J,2}T \parallel u_{1}-u_{2} \parallel_{l},
\end{equation*}
which yields \eqref{datadependence1} and it completes the proof.
\end{proof}

\begin{remark} \label{rem-ft} {\rm We also have existence and uniqueness in $L^2$, that is, given $u(x,t)\in L^2([-1,0]\times [0,T])$ and $v_{0}\in L^2([0,1])$, there exists an unique $v(t,x) \in L^2([0,1]\times [0,T])$, solution to \eqref{nonlocal}. The proof is analogous and hence we omit the details.

In addition, we have a comparison principle, if we have two ordered functions $u \geq \tilde{u}$ and two initial conditions $v_0 \geq \tilde{v_0}$ then
the corresponding solutions verify $v(x,t) \geq \tilde{v} (x,t)$.
}\end{remark}

Now, we need to look back to the local part. Given $v \in C([0,1]\times [0,T])$, we will show that there exists a unique solution $u \in C([-1,0]\times [0,T])$ to \eqref{local}, with $u_0$ as initial condition. We define $H_{2}: Y_{T} \to X_{T}$ as 
the solution operator $H_{2}(v)=u$ and again we prove continuity of this operator.

\begin{lemma} Fix $T>0$.
Given $v(x,t)$ $\in$ $C([0,1]\times [0,T])$ and $u_{0}\in C([-1,0])$, there exists an unique $u(x,t)\in C([-1,0]\times [0,T])$, solution to \eqref{local}. 
Moreover, if $u_{1}$ and $u_{2}$ are the solutions corresponding to $v_{1}$ and $v_{2}$ then 
\begin{equation}\label{datadependence1.99}
\parallel u_{1}-u_{2}\parallel_{l} \leq C_{2} \parallel v_{1}-v_{2}\parallel_{l}.
\end{equation}
\end{lemma}

\begin{proof}
It is well known, see \cite{evans}, that, given $v(t,x)$ $\in$ $C([0,1]\times [0,T])$ and $u_{0}\in C([-1,0])$, the problem \eqref{local} has an unique solution
$u(t,x)\in C([-1,0]\times [0,T])$. 
Therefore, the operator $H_2$ is well defined.

To show the bound \eqref{datadependence1.99} we will use a comparison argument.

Before we start with the argument, we will make some observations that can simplify our problem. First, note that, due to the symmetry
of the kernel and the fact that $\int_{-1}^{1}J(r)dr=1$, we have
\begin{equation}\label{constant}
C_{2}=C_{J,2}\int_{-1}^{0}\int_{0}^{1}J(x-y) dydx = \frac{C_{J,2}}{2}.
\end{equation}

Now, to obtain the estimate, let us consider $ z=u_{1}-u_{2}$, where both $u_{1}$ and $u_{2}$ satisfy \eqref{local} with the same initial condition $u_{0}(x)$ and two different functions $v_{1}$ and $v_{2}$, respectively. Then $z(x,t)$ is a solution to the following problem,
\begin{align}\label{variablez}
\begin{cases}
 \displaystyle   \frac{\partial z}{\partial t}(x,t)  =  \frac{\partial^2 z}{\partial x^2}(x,t),  \\[10pt]
 \displaystyle   \frac{\partial z}{\partial x}(-1,t) =  0,  \\[10pt]
  \displaystyle   \frac{\partial z}{\partial x}(0,t) = c_{J,2}\int_{-1}^{0}\int_{0}^{1}J(x-y)\left[v_1(y,t)-v_2(y,t)-(u_1(0,t)-u_{2}(0,t))\right]dydx, \\[10pt]
     z(x,0)  = 0.
    \end{cases}
    \end{align}

Using \eqref{constant}, we obtain 
\begin{align*}
& \left|   C_{J,2}\int_{-1}^{0}\int_{0}^{1}J(x-y)(v_{1}(y,t)-v_{2}(y,t))dydx\right| \\
& \qquad \leq C_{J,2}\int_{-1}^{0}\int_{0}^{1}J(x-y) |v_{1}(y,t)-v_{2}(y,t)|dy dx \\\nonumber
& \qquad \leq C_{J,2}\int_{-1}^{0}\int_{0}^{1}J(x-y) \max_{t \in [0,T]} \max_{y \in [0,1]}\Big|v_{1}(y,t)-v_{2}(y,t)\Big|dy dx \\\nonumber
& \qquad = C_{2} \parallel v_{1}-v_{2} \parallel_{nl}.
\end{align*}

Hence, if we define 
\begin{equation}\label{variablechange}
w(t,x) = \frac{z(t,x)}{ C_{2} \parallel v_{1}-v_{2} \parallel_{nl}},
\end{equation}
the solution $w(x,t)$ satisfies the following problem
\begin{align}\label{newlocal}
\begin{cases}
 \displaystyle   \frac{\partial w}{\partial t}(x,t)  =  \frac{\partial^2 w}{\partial x^2}(x,t),  \\[10pt]
 \displaystyle   \frac{\partial w}{\partial x}(-1,t) =  0,  \\[10pt]
  \displaystyle   \frac{\partial w}{\partial x}(0,t)  \leq - C_{2} w(0,t) + 1 \quad \text{and} \quad 
 \displaystyle   \frac{\partial w}{\partial x}(0,t)  \geq - C_{2} w(0,t)- 1, \\[10pt]
     w(x,0)  = 0.
    \end{cases}
    \end{align}

Then, to obtain the desired estimates, we focus on to obtain bounds for $w$, a function that verifies the problem \eqref{newlocal}, using the comparison principle. 
Recall that a function $\overline{w}(x,t)$ is called a supersolution for the problem \eqref{newlocal} if it satisfies
\begin{align*}\label{newlocal}
\begin{cases}
 \displaystyle   \frac{\partial \overline{w}}{\partial t}(x,t)  \geq  \frac{\partial^2 w}{\partial x^2}(x,t),  \\[10pt]
 \displaystyle   \frac{\partial \overline{w} }{\partial x}(-1,t) \geq  0,  \\[10pt]
  \displaystyle   \frac{\partial \overline{w}}{\partial x}(0,t)  \leq - C_{2} w(0,t) + 1, \\[10pt]
     \overline{w} (x,0)  = 0.
    \end{cases}
    \end{align*}
Respectively, a function $\underline{w}(x,t)$ is called a subsolution if, it satisfies the reverse inequalities.

Let us introduce an auxiliary function. Given $\xi < 0$ and $0<a<1$ we can define
\begin{equation}\label{auxiliaryfunction}
g(\xi)=\frac{1}{a}f(\xi a), \quad \text{with} \quad g'(0)=f'(0)=1.
\end{equation}
Here, the function $f$ is chosen such that the following conditions hold:

Given $\xi_{0}>1$, $f$ is increasing in $(-\xi_{0},0]$, $C^2(-\xi_{0},0)$, and $f\equiv 1$ in $(- \infty, -\xi_{0}]$. 

Let us fix $T<\frac{a^2}{2\xi_{0}^2}$. For each $t \in [0,T]$ and $x \in [-1,0]$ we define
\begin{equation*}
\overline{w}(x,t)=(T+t)^{1/2}g\left(\frac{x}{(T+t)^{1/2}}\right),
\end{equation*} 
as $g$ given by \eqref{auxiliaryfunction}.

We aim to verify that $\overline{w}$ is a supersolution for \eqref{newlocal}. 

\begin{itemize}
\item[i)] We want to prove that $$\frac{\partial \overline{w}}{\partial t} \geq \frac{\partial^2 \overline{w}}{\partial x^2}.$$
Differentiating $\overline{w}$ with respect to $t$ and $x$, 
we obtain that it is enough to have
$$
\frac{1}{2} g\left(\frac{x}{(T+t)^{1/2}}\right) - \frac{1}{2}x(T+t)^{-1/2} g'\left(\frac{x}{(T+t)^{1/2}}\right) \geq g''\left(\frac{x}{(T+t)^{1/2}}\right).
$$
Observe that, since $x \in [-1,0]$ and $g'\left(\frac{x}{(T+t)^{1/2}}\right)>0$, we only need to verify that
\begin{equation}\label{ineq-I}
\frac{1}{2} g\left(\frac{x}{(T+t)^{1/2}}\right)\geq g''\left(\frac{x}{(T+t)^{1/2}}\right).
\end{equation}
To deal with this, let us call $\eta= \frac{x}{(T+t)^{1/2}}$. According to the definition of $g$, to prove \eqref{ineq-I} is equivalent to prove
 \begin{equation}\label{ineq-Ia}
\frac{1}{2a}f(a\eta)\geq a f''(a\eta).
\end{equation}
We know that, for each $\xi \leq 0$ and $0<a<1$,
\begin{align*}
\frac{f(\xi a)}{2}=
\begin{cases}
 \displaystyle   1/2, \qquad \text{if} \ \xi a<-\xi_{0},  \\[10pt]
 \displaystyle   \frac{f(\xi a)}{2}, \qquad \text{if} \ -\xi_{0} \leq \xi a<0.  
    \end{cases}
    \end{align*}
Moreover, as $f \in C^2(-\xi_{0},0)$ and increasing in the same interval, we obtain
\begin{align*}
f''(\xi a) \leq
\begin{cases}
 \displaystyle   0, \qquad \text{if} \ \xi a<-\xi_{0},  \\[10pt]
 \displaystyle   M, \qquad \text{if} \ -\xi_{0} \leq \xi a<0,  \\
    \end{cases}
    \end{align*}
where $M=\max_{-\xi_{0}\leq \xi \leq 0}|f''(\xi)|$. 

Hence, given $M$, we can choose $0<a<1$ in order to have the estimate $\frac{1}{2}\geq Ma^{2}$. With this in mind we are able to  verify \eqref{ineq-Ia}. Indeed, 
\begin{itemize}
\item[a)] If $ -\xi_{0} \leq \xi a<0$, it follows that
$$
\frac{f(\xi a)}{2}\geq \frac{1}{2}\geq Ma^{2}\geq f''(\xi a) a^{2}.
$$
\item[b)] If $\xi a<-\xi_{0}$, we have that
$$
\frac{f(\xi a)}{2}\geq \frac{1}{2}\geq 0 \geq f''(\xi a) a^{2}.
$$
\end{itemize}

\item[ii)] We want to verify that $\overline{w}$ satisfies $$\frac{\partial \overline{w}}{\partial x}(-1,t) \leq 0.$$ 
At $x=-1$ we have,
\begin{equation}\label{iicondition}
\frac{\partial \overline{w}}{\partial x}(-1,t) =g'\left(\frac{-1}{(T+t)^{1/2}}\right)=f'\left(\frac{-a}{(T+t)^{1/2}}\right).
\end{equation}
We know that $f\equiv 1$ in $(-\infty, -\xi_{0})$. Then, taking $T<\frac{a^2}{2\xi_{0}^2}$, we obtain that $\frac{-a}{(T+t)^{1/2}}<-\xi_{0}$ and therefore 
$$f'\left(\frac{-a}{(T+t)^{1/2}}\right)=0$$ and we obtain \eqref{iicondition} as desired.

\item[iii)] We want to check that $$\frac{\partial \overline{w}}{\partial x}(0,t) \geq - C_{2}\overline{w}(0,t)+1.$$ 
Differentiating $\overline{w}$ with respect to $x$, we need to prove that it holds that
$$
g'(0) \geq - C_{2}(T+t)^{1/2} g(0)+1.
$$
As $g'(0)=f'(0)=1$ and $g(0)\geq 1$, we get 
$$1 \geq - C_{2}(T+t)^{1/2} g(0)+1,$$
which prove the item $iii).$

 \item[iv)] Finally, we aim to verify that $\overline{w}(x,0) \geq 0$. 
 
 Indeed, we have
 \begin{equation*}
 \overline{w}(x,0) = (T)^{1/2}\underbrace{g\left(\frac{x}{T^{1/2}}\right)}_{>0}>0.
 \end{equation*}
\end{itemize}

With these four items we proved that $\overline{w}$ is a supersolution of \eqref{newlocal}. 

The fact that $$\underline{w}(t,x)=-(T+t)^{1/2}g\left(\frac{x}{(T+t)^{1/2}}\right)$$ is a subsolution for the problem \eqref{newlocal}
can be proved analogously.

So, by the comparison principle, the solution $w(x,t)$ of the problem \eqref{newlocal}, verifies
$$
\underline{w}(x,t)\leq w(x,t)\leq \overline{w}(x,t).
$$ 
Hence, we get the estimate
$$
|\overline{w}(x,t)| \leq \max_{x \in [-1,0]} \max_{t \in [0, T]}\left|(T+t)^{1/2}g\left(\frac{x}{(T+t)^{1/2}}\right)\right|= \frac{1}{a}(2T)^{1/2}< \frac{1}{\xi_{o}}<1. 
$$

Therefore, going back to our original variable $z$ we have obtained that
$$
 \frac{|z(x,t)|}{ C_{2} \parallel v_{1}-v_{2} \parallel_{nl}} = \frac{|u_{1}-u_{2}|}{ C_{2} \parallel v_{1}-v_{2} \parallel_{nl}}=w(x,t) \leq \overline{w} (x,t) \leq 1,
$$
which implies that
$$
||u_{1}-u_{2}||_{l} \leq C_{2} \parallel v_{1}-v_{2} \parallel_{nl}.
$$
The proof is complete.
\end{proof}

\begin{remark} \label{rem-km} {\rm In this case we also have existence and uniqueness in $L^2$, as in Remark \ref{rem-ft}. 
Given $v(x,t)\in L^2([0,1]\times [0,T])$ and $u_{0}\in L^2([-1,0])$, there exists an unique $u(x,t) \in C^1([-1,0];L^2[0,T])$, solution to \eqref{local}.

Again, we have a comparison principle, if we have two ordered functions $v \geq \tilde{v}$ and two initial conditions $u_0 \geq \tilde{u_0}$ then
the corresponding solutions verify $u(x,t) \geq \tilde{u} (x,t)$.
}\end{remark}

Finally, combining the two lemmas, we get the following theorem.

\begin{theorem}
Given $w_{0} \in C([-1,1])$ (or given $w_{0} \in L^2([-1,1])$), there exists a unique solution to problem \eqref{local}--\eqref{nonlocal}, which has $w_{0}$ as initial condition.
\end{theorem}
\begin{proof}
Let $T\in \left(0,\frac{1}{2C_{J,1}+C_{J,2}}\right)$. 
We consider the operator $H:X_T \mapsto X_T$ given by 
$$H(u):=H_{2}(H_{1}(u)) = H_2 (v),$$ 
and we obtain, from our previous results,
\begin{equation*}
\parallel H_{2}(H_{1}(u_1)) - H_{2}(H_{1}(u_2)) \parallel_{l} = \parallel H_2 (v_1)- H_2(v_2) \parallel_{l} \leq C_2 \parallel v_{1}-v_{2} \parallel_{nl} \leq C_2 \frac{C_{J,2} T}{1- (2 C_{J,1}+C_{J,2})T} \parallel u_{1}-u_{2} \parallel_{l}, 
\end{equation*}
which proves that $H$ is a strict contraction for $T$ small enough. Therefore, there is a fixed point
$$
u=H(u)
$$ 
that gives us a unique solution $(u,v=H_1(u))$ in $(0,T)$. 
Since $T$ can be chosen independently of the initial condition, the fixed point argument can be iterated to obtain a global solution for our problem.
\end{proof}

\subsection{Conservation of mass}
As we expected the model \eqref{complete} preserves the total mass of the solution.

\begin{theorem}
The solution $w$ of the problem \eqref{complete}, with initial condition $w_{0} \in C([-1,1])$ satisfies 
\begin{equation*}
\int_{-1}^{1}w(x,t) dx=\int_{-1}^{1}w_{0}(x) dx,\quad  \mbox{for every } t\geq 0.
\end{equation*}
\end{theorem}
\begin{proof}
First we observe that we have 
\begin{equation*}
\int_{-1}^{1}w(x,t)dx=\int_{-1}^{0}u(x,t)dx + \int_{0}^{1}v(x,t)dx,
\end{equation*}
and also for $w_{0}$
\begin{equation*}
\int_{-1}^{1}w_{0}=\int_{-1}^{0}u_{0} + \int_{0}^{1}v_{0}.
\end{equation*}

From the symmetry of the kernel and Fubbini's theorem, we obtain
\begin{align*}
\displaystyle 
& \frac{\partial}{\partial {t}}\left(\int_{-1}^{1}w(x,t)dx\right) \\[6pt]
& = \int_{-1}^{0} \frac{\partial^2 u}{\partial x^2}(x,t)dx+C_{J,1}\int_{0}^{1}\int_{0}^{1}J(x-y)(v(y,t)-v(x,t))dydx \\[6pt]
& \qquad -C_{J,2}\int_{0}^{1}\int_{-1}^{0}J(x-y)v(x,t)dydx+ C_{J,2}u(0,t)\int_{0}^{1}\int_{-1}^{0}J(x-y)dydx \\[6pt]
& =  \frac{\partial u}{\partial {x}}(0,t)-  \frac{\partial u}{\partial {x}}(-1,t)
 -C_{J,2}\int_{0}^{1}\int_{-1}^{0}J(x-y)v(x,t)dydx  + C_{J,2}u(0,t)\int_{0}^{1}\int_{-1}^{0}J(x-y)dydx \\[6pt]
& = 0.
\end{align*}
This shows that the total mass is independent of $t$.
\end{proof}

\subsection{Comparison principle}
Thanks to the linearity of the operator, if we have two solutions to the local/nonlocal problem \eqref{local}-\eqref{nonlocal}, then the difference is also a solution. 
From the fixed point construction of the solution, given a nonnegative initial data $u_{0}, v_{0}$, the solution
is nonnegative for every positive time (this follows from our construction of the solutions as a fixed point
and Remarks \ref{rem-ft} and \ref{rem-km}). Therefore, we have the following result:
\begin{proposition} Let $u_0 \geq \tilde{u_0}$ and $v_0 \geq \tilde{v_0}$ then the corresponding solutions 
to the local/nonlocal problem \eqref{local}--\eqref{nonlocal} verify
$$
u(x,t) \geq \tilde{u} (x,t), \qquad v(x,t) \geq \tilde{v} (x,t),
$$
for every $t\geq 0$.
\end{proposition}

To go one step further, let us define what we understand by sub and supersolutions.

\begin{definition} {\rm
The functions $\overline{u}$ and $\overline{v}$ are called supersolutions of the problem \eqref{local}-\eqref{nonlocal}, in $ [-1,1]\times [0,T]$ if, $\overline{u},\overline{v}$ verify
\begin{align*}
\begin{cases}
 \displaystyle    \frac{\partial \overline{u}}{\partial {t}} (t,x) \geq   \frac{\partial^2 \overline{u}}{\partial {x^2}}(t,x), \quad x \in (-1,0), \quad t>0, \\[10pt]
 \displaystyle    \frac{\partial \overline{u}}{\partial {x}}(t,-1) \leq  0, \, t>0, \\[10pt]
  \displaystyle    \frac{\partial \overline{u}}{\partial {x}}(t,0) \geq  C_{J,2} \int_{-1}^{0}\int_{0}^{1}J(x-y) (\overline{v}(y,t)-\overline{u}(t,0))dydx, \quad t>0, 
  \\[10pt]
   \displaystyle    \frac{\partial \overline{v}}{\partial {t}}(t,x) \geq  C_{J,1} \int_{0}^{1}J(x-y) (\overline{v}(y,t)-\overline{v}(x,t))dy 
   -C_{J,2} \int_{-1}^{0}J(x-y) (\overline{v}(x,t)-\overline{u}(t,0))dy, \quad x \in (0,1), \, t>0,\\[10pt]
     \overline{u}(0,x)  \geq u_{0}(x), \quad x \in (-1,0), \qquad 
      \overline{v}(0,x)  \geq v_{0}(x), \quad x \in (0,1).
    \end{cases}
    \end{align*}    
As usual, subsolutions, $\underline{u},\underline{v}$, are defined analogously by reversing the inequalities. }
\end{definition}

\begin{theorem} Let $\overline{u}$ and $\overline{v}$ be supersolutions of the problem \eqref{local}-\eqref{nonlocal}.
If $(\overline{u}_{0},\overline{v}_{0}) \geq 0$ then 
$$\overline{u} (x,t) \geq 0, \qquad \mbox{and} \qquad \overline{v}(x,t)\geq 0,$$
for every $t>0$. 

Moreover, given $(\overline{u},\overline{v}) $ a supersolution and $(\underline{u},\underline{v})$ a subsolution with 
$\overline{u}_{0} \geq \underline{u}_{0}$ and $\overline{v}_{0} \geq \underline{v}_{0}$, then, 
$$\overline{u} (x,t) \geq \underline{u} (x,t) \qquad \mbox{and} \qquad 
\overline{v} (x,t) \geq \underline{v} (x,t), $$ 
for every $t>0$.
\end{theorem}
\begin{proof}
Let us define
\begin{align*}
\begin{cases}
\displaystyle w = \overline{u}-\underline{u}, \\
\displaystyle z = \overline{v}-\underline{v}.
\end{cases}
\end{align*}
We have that $w$ and $z$ are supersolutions with $w(x,0) \geq 0$ and $z(x,0)\geq 0$. In fact, we have that
\begin{itemize}
\item[i)] $\displaystyle  \frac{\partial w}{\partial t}=  \frac{\partial \overline{u}}{\partial t} - 
 \frac{\partial \underline{u}}{\partial t}  \geq \frac{\partial^2 \overline{u}}{\partial x^2} -
 \frac{\partial^2 \underline{u}}{\partial x^2}  = \frac{\partial^2 w}{\partial x^2} ;$
\item[ii)] $\displaystyle \frac{\partial w}{\partial x}(-1,t)=
\frac{\partial \overline{u}}{\partial x} (-1,t)- \frac{\partial \underline{u}}{\partial x} (-1,t) \leq 0;$
\item[iii)] \begin{align*}
\frac{\partial w}{\partial x}(0,t) & =
\frac{\partial \overline{u}}{\partial x} (0,t)- \frac{\partial \underline{u}}{\partial x} (0,t) \\[10pt]
& \leq C_{J,2}\int_{-1}^{0}\int_{0}^{1}J(x-y)(\overline{v}(y,t)-\overline{u}(0,t))dydx 
-C_{J,2}\int_{-1}^{0}\int_{0}^{1}J(x-y)(\underline{v}(y,t)-\underline{u}(0,t))dydx \\[10pt]
& = C_{J,2}\int_{-1}^{0}\int_{0}^{1}J(x-y)(z(y,t)-w(0,t))dydx;
\end{align*}
\item[iv)] \begin{align*}
\frac{\partial z}{\partial t} & = \frac{\partial \overline{v}}{\partial t} - \frac{\partial \underline{v}}{\partial t}   \\[10pt]
& \geq C_{J,1}\int_{0}^{1}J(x-y)(\overline{v}(y,t)-\overline{v}(x,t))dy-C_{J,2}\int_{-1}^{0}J(x-y)(\overline{v}(x,t)-\overline{u}(0,t))dydx \\[10pt]
& \qquad - \left(C_{J,1}\int_{0}^{1}J(x-y)(\underline{v}(y,t)-\underline{v}(x,t))dy-C_{J,2}\int_{-1}^{0}J(x-y)(\underline{v}(x,t)-\underline{u}(0,t))dydx \right)
\\[10pt]
&= C_{J,1}\int_{0}^{1}J(x-y)(z(y,t)-z(x,t))dy-C_{J,2}\int_{-1}^{0}J(x-y)(z(x,t)-w(0,t))dydx.
\end{align*}
\end{itemize}
Now, we want to show that $w \geq 0$ and $z \geq 0$, for all $t>0$, which implies $\overline{u}\geq \underline{u}$ and $\overline{v}\geq \underline{v}$. To show this we aim to conclude that  its negative parts are identically zero, $w_{-} \equiv 0$ and $z_{-} \equiv 0$.
Take $\varphi = w_{-} \geq 0$ and $\psi = z_{-} \geq 0$ as test functions 
(multiply the previous inequalities by them and integrate). We obtain,
\begin{align*}
0 & \leq \left(\int_{-1}^{0} \frac{\partial w}{\partial t}  \varphi + \int_{0}^{1} \frac{\partial z}{\partial t}  \psi\right) \\[10pt]
& = -\int_{-1}^{0} \frac{\partial w}{\partial x}  \frac{\partial \varphi}{\partial x} dx
-\frac{C_{J,1}}{2}\int_{0}^{1}\int_{0}^{1}J(x-y)(z(y)-z(x))(\psi(y)-\psi(x))dydx \\[10pt]
& \qquad -C_{J,2}\int_{0}^{1}\int_{-1}^{0}J(x-y)(z(x)-w(0))(\psi(x)-\varphi(0))dydx \\[10pt]
& = -2 E(w_{-},z_{-}) \leq 0.
\end{align*}
Therefore $E(w_{-},z_{-})=0$, which implies $w_{-} \equiv 0$ and $z_{-} \equiv 0$, as we wanted to show.
\end{proof}

\subsection{Asymptotic decay}
Now we study the asymptotic behavior as $t \to \infty$. We start by analyzing
the corresponding stationary problem.

First, let us observe that for any constant $k$,
$u=v=k$, is a solution to the problem \eqref{local}--\eqref{nonlocal}. Besides, this constant
solution is a minimizer of the energy (a simple inspection of the energy shows more, every
minimizer is constant in the whole domain $(-1,1)$).

Let us take $\beta_{1}$ as
\begin{align}\label{eigenvalue}
\displaystyle 
    0<\beta_{1}=\inf_{u,v : \int_{-1}^{0} u+\int_{0}^{1} v=0}\frac{
    \displaystyle E(u,v)}{\displaystyle  \int_{-1}^{0}(u(x))^2 dx+\int_{0}^{1}(v(x))^2 dx}.
\end{align}
Now our goal is to show that $\beta_1$ is strictly positive.

\begin{lemma}\label{GQR}
Let $\beta_{1}$ br given by \eqref{eigenvalue}, then
$$
\beta_{1}>0,
$$ 
and is such that 
\begin{align}\label{ener.ineq}
\displaystyle 
    E(u,v) \geq \beta_{1}  \int_{-1}^{0}(u(x))^2 dx+\int_{0}^{1}(v(x))^2 dx,
\end{align}
for every $(u,v)$ such that $\int_{-1}^{0}u+\int_{0}^{1}v=0$.
\end{lemma}
\begin{proof}
Let us argue by contradiction. Suppose that \eqref{ener.ineq} is false. Then there exists sequences $\left\{u_{n}\right\} \in H^1(-1,0)$ and $\left\{v_{n}\right\} \in L^2(0,1)$ such that 
\begin{itemize}
\item[i)] $\displaystyle \int_{-1}^{0}u_{n}+\int_{0}^{1}v_{n}=0,$
\item[ii)] $\displaystyle \int_{-1}^{0}(u_{n})^2+\int_{0}^{1}(v_{n})^2=1$,
\end{itemize}
and 
$$\frac{1}{2}\int_{-1}^{0}\Big| \frac{\partial u_n}{\partial x}\Big|^2 dx + \frac{C_{J,1}}{4}\int_{0}^{1}\int_{0}^{1}J(x-y)\left(v_n(y)-v_n(x)\right)^2 dydx 
+  \frac{C_{J,2}}{2}\int_{0}^{1}\int_{-1}^{0}J(x-y)dy\left(v_n(x)-u_n(0)\right)^2 dx \leq \frac{1}{n}.
$$
Consequently, taking the limit in $n$, we obtain
\begin{equation}\label{a}
    \lim_{n}\left(\frac{1}{2}\int_{-1}^0 \Big| \frac{\partial u_n}{\partial x}\Big|^2\right)=0,
\end{equation}
\begin{equation}\label{b}
    \lim_{n}\left(\frac{C_{J,1}}{4}\int_{0}^{1} \int_{0}^{1} J(x-y) (v_n(y) - v_n(x))^2 dy dx\right)=0,
\end{equation}
and
\begin{equation}\label{c}
    \lim_{n}\left(\frac{C_{J,2}}{2}\int_{-1}^{0} \int_{0}^{1} J(x-y) (u_n(0) - v_n(x))^2 dx dy\right)=0.
\end{equation}
From $ii)$ we have that $\int_{-1}^{0}(u_{n})^2 \leq 1$. Then, $u_{n}$ is bounded in $H^1(-1,0)$, and hence there exists a subsequence $\left\{u_{nj}\right\} \in H^1(-1,0)$, which weakly converges for a limit $u \in H^1(-1,0)$. 
From this weak convergence follows the convergence of $\left\{u_{nj}\right\} \to u$ in $L^2(-1,0)$ and the uniform convergence in $[-1,0]$ to $u \in H^1(-1,0)$. Moreover, as $\frac{1}{2}\int_{-1}^0 ((u_{n})_{x})^{2} \to 0$ we have that the limit $u$ is a constant, $u=k_{1}$. Note that $u_{nj}(0) \to k_{1}$.

Also from $ii)$ we have that $\int_{0}^{1}v_{n}^2 \leq 1$. Since $\left\{v_{n}\right\}$ is bounded in $L^2(0,1)$, there exists a subsequence $\left\{v_{nj}\right\}$, which weakly converge for some limit $v$ in $L^2(0,1)$.
Now, we observe that
\begin{equation*}
\int_{0}^{1}|v_{nj}| \leq \left(\int_{0}^{1}v_{nj}^2\right)^{1/2} \leq 1.
\end{equation*}
Since $\left\{v_{n}\right\}$ is bounded in $L^2(0,1)$, there exists a subsequence such that $k_{n}=\int_{0}^{1}v_{n}$ converges to some limit $k_{2}$.

Consider $z_{n} = v_{n} - k_{n}$. We have (by Lemma 4.2, see \cite{ChChRo})
\begin{align*}
& \frac{C_{J,1}}{4}\int_{0}^{1}\int_{0}^{1}J(x-y)|z_{n}(y)-z_{n}(x)|^2dydx  =\frac{C_{J,1}}{4}\int_{0}^{1}\int_{0}^{1}J(x-y)|v_{n}(y)-v_{n}(x)|^2 dydx \to 0.
\end{align*}
In fact, applying the result in \cite{ChChRo}, (notice that we have $\int_{0}^{1}z_{nkj}=0$),
\begin{align*}
\frac{C_{J,1}}{4}\int_{0}^{1}\int_{0}^{1}J(x-y)(z_{n}(y)-z_{n}(x))^2 dxdy \geq c \int_{0}^{1}z_{n}^2.
\end{align*}
From this, we obtain that $z_{n} \to 0$ in $L^2(0,1)$, so
\begin{align*}
\int_{0}^{1}|v_{n}(x)-k_{n}|^2 \to 0,
\end{align*}
but $k_{n} \to k_{2}$, and then $v_{n} \to k_{2}$ in $L^2(0,1)$.

Finally, from \eqref{c} we have $$\frac{C_{J,2}}{2}\int_{-1}^{0} \int_{0}^{1} J(x-y) (u_n(0) - v_n(x))^2 dx dy \to 0,$$ 
which leads to
\begin{align*}
\frac{C_{J,2}}{2}\int_{-1}^{0} \int_{0}^{1} J(x-y) (k_{1}-k_{2})^2 dx dy,
\end{align*}
and hence $k_{1}=k_{2}$. On the other hand, by $i)$ we have $\int_{-1}^{0}u_{n}+\int_{0}^{1}v_{n}=0$, so $k_{1}=0$ and $k_{2}=0$, 
that leads to a contradiction passing to the limit in $\int_{-1}^{0}(u_{n})^2+\int_{0}^{1}(v_{n})^2=1$ to obtain
$\int_{-1}^{0}(k_1)^2+\int_{0}^{1}(k_2)^2=1$.
\end{proof}

\begin{remark} {\rm This value $\beta_1$ should be the first nontrivial eigenvalue for our problem
(notice that $\beta =0$ is an eigenvalue with $u=v=cte$ as eigenfunctions). However, due to the
lack of compactness of the nonlocal part, it is not clear that the infimum defining $\beta_1$ is attained.
}
\end{remark}

Now, we are ready to prove the exponential convergence of the solutions to the mean value of the initial datum
as $t\to +\infty$.

\begin{theorem} \label{teo.asymp.intro.909}
Given $w_{0} \in L^2(-1,1)$, the solution to \eqref{complete} with initial condition $w_0$ 
converges to its mean value as $t \to \infty$ with an exponential rate,
\begin{equation}
    \left\| w(\cdot ,t) - \fint w_0 \right\|_{L^2(-1,1)} \leq C(\| w_{0} \|_{L^2(-1,1)}) e^{-\beta_{1} t}, \qquad t>0,
\end{equation}
where $\beta_{1}$ is given by \eqref{eigenvalue} and $C(w_{0})>0$.
\end{theorem}

\begin{proof}
As we know, $u=v=k$, $k$ constant, is a solution of \eqref{local}--\eqref{nonlocal}. In particular $h(x,t)=u(x,t)-k$ and $z(x,t)=v(x,t)-k$ is also a solution. If $k= \int_{-1}^{0}u_{0}+\int_{0}^{1}v_{0}$, 
then $h$ and $z$ satisfy $$\int_{-1}^{0}h(x,t)dx+\int_{0}^{1}z(x,t)dx=0.$$

Let $$f(t)=\frac{1}{2}\int_{-1}^{0}h(x,t)^2dx+\frac{1}{2}\int_{0}^{1}z(x,t)^2dx.$$
Differentiating $f$ with respect to $t$, we obtain
\begin{align*}
& f'(t) =\int_{-1}^{0}h  \frac{\partial h}{\partial t} dx + \int_{0}^{1}z  \frac{\partial z}{\partial t} dx \\[10pt]
& = \int_{-1}^{0}h  \frac{\partial^2 h}{\partial x^2} dx + C_{J,1}\int_{0}^{1}z(x,t)\int_{0}^{1}J(x-y)(z(y,t)-z(x,t))dy dx \\[10pt]
& \qquad - C_{J,2}\int_{0}^{1}z(x,t)\int_{-1}^{0}J(x-y)z(x,t)dydx + C_{J,2}\int_{0}^{1}z(x,t)h(0,t)\int_{-1}^{0}J(x-y)dydx \\[10pt]
& =h(0,t) \frac{\partial h}{\partial x}(0,t)-h(-1,t) \frac{\partial h}{\partial x} (-1,t)-\int_{-1}^{0} \Big|  \frac{\partial h}{\partial x}\Big|^2 dx \\[10pt]
& \qquad +C_{J,1}\int_{0}^{1}\int_{0}^{1}J(x-y)(z(y,t)-z(x,t))z(x,t)dy dx \\[10pt]
&  \qquad - C_{J,2}\int_{0}^{1}\int_{-1}^{0}J(x-y)z(x,t)^2dydx + C_{J,2}\int_{0}^{1}h(0,t)\int_{-1}^{0}J(x-y)z(x,t)dydx. 
\end{align*}
Applying Fubini's theorem and using the symmetry of the kernel we obtain
\begin{align*}
f'(t)= -2 E(h,z) (t).
\end{align*}
Finally, applying Lemma \ref{GQR}, we get
$$
2 E(h,z)  \geq 2\beta_{1}\int_{-1}^{0}h^2+\int_{0}^{1}z^2  \geq 2\beta_{1} f(t),
$$
which implies $$f'(t)\leq -2\beta_{1}f(t).$$ 
Hence,
$$f(t)\leq e^{-2 \beta_{1}t}f(0),$$ where
\begin{align*}
f(0)=\frac{1}{2}\left(\int_{-1}^{0}h_{0}^2dx+\int_{0}^{1}z_{0}^2dx\right) = C(\parallel w_{0}\parallel_{L^2(-1,1)}).
\end{align*}
From this follows that
\begin{align*}
\int_{-1}^{0}|u(t,x)-k|^2dx+\int_{0}^{1}|v(t,x)-k|^2dx \leq C(\parallel w_{0}\parallel_{L^2(-1,1)}) e^{-2 \beta_{1}t} \to 0,
\end{align*}
as $t \to \infty$. In particular, we have that $u \to k$ in $L^2(-1,0)$ and $v \to k$ in $L^2(0,1)$.
\end{proof}

\section{Rescaling the kernel. Convergence to the local problem}
We will prove the strong convergence in $L^2(-1,1)$, uniformly  or bounded times, of the solutions of the rescaled problem (with $J$ as in \eqref{rescale-J})
 to the solution of the local problem \eqref{localcomplete} (the heat equation in the whole domain with homogeneous Newman boundary conditions) 
 using the Brezis-Pazy Theorem trough Mosco's convergence result. To perform this task we need to provide another existence of solutions results for the problem \eqref{local}-\eqref{nonlocal} based on semigroup theory for m-accretive operators.
 
 \subsection{Existence and uniqueness of a mild solution}
{\bf On the concept of solution.}
We will introduce now the concept of solution for the problem \eqref{complete} that we are going to use here. 
We rely on serigroup theory and introduce the operator
\begin{align*}
 \displaystyle 
B_{J}u (x) =
\begin{cases} 
\displaystyle  -  \frac{\partial^2 u}{\partial x^2} (x) \qquad  \text{for} \quad x \in (-1,0),\\[10pt]
\displaystyle  -\frac{C_{J,1}}{2}\int_{0}^{1}J(x-y)(u(y)-u(x))dy+C_{J,2}\int_{-1}^{0}J(x-y)(u(y)-u(0))dy \qquad \text{for} \quad x \in (0,1).
\end{cases}
\end{align*}
Let $$D(B_{J}):=\Big\{ (u,v) : u \in H^2(-1,0), v\in L^2(0,1) \mbox{ with }  \frac{\partial u}{\partial x}(-1) = 0 
\mbox{ and }  \frac{\partial u}{\partial x} (0) = -C_{J,2}\int_{-1}^{0}J(x-y)(v(y)-u(0))dy\Big\},
$$ 
be the domain of the operator, and
$$
B_{J} : D(B_{J}) \subset L^2(-1,1) \mapsto L^2(-1,1).
$$ 

Now, according to \cite{andreu2010nonlocal}, we can define a mild solution in $L^2(-1,1)$, of the abstract Cauchy problem by:
\begin{align}\label{semigroup}
\begin{cases}
&u'(t)=B_{J}(u(t)), \quad t>0\\
&u(0)=u_{0}.
\end{cases}
\end{align}
Moreover, given an initial condition in the domain of the operator, there exists a unique strong solution for this problem, given by the semigroup related to $B_{J}$, see \citep{andreu2010nonlocal,brezis} for more details.

Following the ideas presented in \cite{andreu2010nonlocal}, we will see that the operator $B_{J}$ is completely accretive in $L^2(-1,1)$ and satisfies the range condition, $L^2(-1,1)\subset R(I+B_{J})$, and hence $B_J$ is $m-$completely accretive in $L^2(-1,1)$. The range condition implies that for any $f \in L^2(-1,1)$ there exists $u \in D(B_{J})$ such that, $u+B_{J}(u)=f$, and the resolvent $(I+B_{J})^{-1}$ is a contraction in $L^2(-1,1)$. With this in mind, by the Crandall-Ligget's Theorem we will obtain the existence and uniqueness of a mild solution for the coupled local/nonlocal evolution problem.

\begin{theorem}
Given and initial condition $w_{0}$ $\in L^{2}(-1,1)$, there exists a mild solution $w$ of the problem \eqref{complete} that is a contraction in $L^2-$norm.
\end{theorem}
\begin{proof}
It is enough to show that the operator $B_{J}$ is completely accretive in $L^2(-1,1)$ and satisfies the range condition, $L^2(-1,1)\subset R(I+B_{J})$.
Consider the set 
\begin{equation*}
P_{0}=\left\{q \in C^{\infty}(-1,1): 0 \leq q \leq 1, \quad \text{supp($q'$) is compact and $ 0 \not \in$ supp($q$)} \right\}.
\end{equation*}
To show the operator $B_{J}$ is completely accretive is equivalent to show that, 
given $w_{1}, w_{2} \in D(B_{J})$, and $q(w_{1}-w_{2})$, as a function test, we have that
\begin{equation}\label{monotonocy}
\int_{-1}^{1}(B_{J}(w_{1}(x))-B_{J}(w_{2}(x)))q(w_{1}(x)-w_{2}(x))dx \geq 0.
\end{equation}
Using the weak form of the operator we get
\begin{align*}
&\int_{-1}^{1}(B_{J}(w_{1}(x))-B_{J}(w_{2}(x)))q(w_{1}(x)-w_{2}(x))dx \\[10pt] 
&= \int_{-1}^{0}\frac{\partial (w_{1}-w_{2})}{\partial x} \frac{\partial [q(w_{1}-w_{2}(x))]}{\partial x}dx \\[10pt]
& \qquad +\frac{C_{J,1}}{2}\int_{0}^{1}\int_{0}^{1}J(x-y)[(w_{1}-w_{2})(y)-(w_{1}-w_{2})(x)]  [q(w_{1}(y)-w_{2}(y))-q(w_{1}(x)-w_{2}(x))]dydx\\[10pt]
& \qquad +C_{J,2}\int_{0}^{1}\int_{-1}^{0}J(x-y)[(w_{1}-w_{2})(x)-(w_{1}-w_{2})(0)]
[q(w_{1}(x)-w_{2}(x))-q(w_{1}(0)-w_{2}(0))]dydx.
\end{align*}
Since $J \geq 0$, using the Mean Value Theorem, we obtain that the inequality \eqref{monotonocy} holds.

To see that $B_{J}$ is m–completely accretive in $L^{2}(-1,1)$ we need to show that it satisfies the range condition
\begin{equation*}
L^2(-1,1)\subset R(I+B_{J}).
\end{equation*}
Given $f \in L^2(-1,1)$, we consider the variational problem
\begin{equation}\label{variational}
I[u]= \min_{u \in L^2(-1,1)}\left\{\frac{1}{2}\int_{-1}^{1}u^2+E(u)-\int_{-1}^{1}fu\right\}.
\end{equation}
The existence of a unique minimizer $u$, of the variational problem \eqref{variational}, is proved using the direct method in the calculus of variations. This operator is continuous, monotone and coercive in $L^2(-1,1)$. Indeed, using Young's inequality, we obtain
 \begin{align}\label{coercive}
\frac{1}{2}\int_{-1}^{1}u^2+E(u)-\int_{-1}^{1}fu & \geq \frac{1}{2}\int_{-1}^{1}u^2+E(u) -\left(\int_{-1}^{1}f^2\right)^{1/2}\left(\int_{-1}^{1}u^2\right)^{1/2} 
 \geq \frac{3}{8}\int_{-1}^{1}u^2+E(u)-C,
\end{align}
and then
\begin{equation*}
\lim_{{\parallel u \parallel}_{L^2(-1,1)} \to \infty}\frac{I(u)}{\parallel u \parallel}_{L^2(-1,1)} \geq \lim_{{\parallel u \parallel}_{L^2(-1,1)} \to \infty} \frac{\left(\frac{3}{8}\parallel u \parallel _{L^2(-1,1)}+E(u)-C\right)}{\parallel u \parallel_{L^2(-1,1)}} = + \infty.
\end{equation*} 
Then, from \cite{evans}, there exists a minimizing sequence $\{u_{n}\}$ in $H^1(-1,0)\cap L^2(-1,1)$, with $n \in \mathbb{N}$, such that 
\begin{equation*}
\frac{1}{2}\int_{-1}^{1}u_{n}^2+E(u_{n})-\int_{-1}^{1}fu_{n}\leq C, \quad \forall \quad n \in \mathbb{N}.
\end{equation*} 

Therefore $\parallel u_{n}\parallel_{L^2(-1,1)} \leq M$ and $\parallel u_{n}\parallel_{H^1(-1,0)} \leq M$, for all $n \in \mathbb{N}$.
Hence, by the compact embedding theorem [\cite{evans}, Rellich-Kondrachov Compactness Theorem], we can assume, taking a subsequence if necessary, that $u_{n} \rightharpoonup u$ in $L^2(-1,1)$, $u_{n} \to u$ in $L^2(-1,0)$, and by the reflexivity of $H^1(-1,0)$, we get that $u \in H^1(-1,0)$. 

According to \cite{evans}, as the functional $I(u)$ is bounded and convex, it follows that $I(u)$ is weakly lower semicontinuous, 
\begin{equation}\label{weaklylowersem}
I(u)\leq \liminf_{n \to \infty}I(u_{n}).
\end{equation}
Thanks to \eqref{weaklylowersem} and \eqref{coercive}, we can conclude that $u$ is actually a minimizer of the variational problem \eqref{variational}. The uniqueness follows by the strict convexity of the functional.
\end{proof}

\begin{remark} {\rm One can also show existence and uniqueness using Hille-Yosida Theorem. 
In fact, one can show that $B_J$ is closed, its domain $D(B_J)$ is dense in $L^2(-1,1)$ and it holds that
for every $\lambda>0$,
$$
\| (\lambda - B_J)^{-1} \|_{L^2 (-1,1)} \leq \frac1{\lambda}.
$$
}
\end{remark}

 The energy functional associated to the rescaled problem is given by
 \begin{align}
        E^{\varepsilon}(w^{\varepsilon})  := & \frac{1}{2}\int_{-1}^{0}\Big| \frac{\partial u^{\varepsilon}}{\partial x}\Big|^2 + \frac{C_{J,1}}{4\varepsilon^{3}}\int_{0}^{1}\int_{0}^{1}J^{\varepsilon}(x-y)\left(v^{\varepsilon}(y)-v^{\varepsilon}(x)\right)^2 dydx \\ \nonumber
   & + \frac{C_{J,2}}{2\varepsilon^{3}}\int_{-1}^{0}\int_{0}^{1}J^{\varepsilon}(x-y)\left(u^{\varepsilon}(0)-v^{\varepsilon}(x)\right)^2 dxdy,
 \end{align}
 if $w$ $\in$ $D(E^{\varepsilon}):=H^1(-1,0)\times L^2(0,1)$, and $E^{\varepsilon}(w):=\infty$ if not. Analogously, we define the limit energy functional as
 \begin{equation}
     E(w):= \frac{1}{2}\int_{-1}^{1}\Big| \frac{\partial w}{\partial x}\Big|^{2}  dx,
 \end{equation}
if $w$ $\in$ $D(E):=H^1(-1,1)$, and $E^{\varepsilon}(w):=\infty$ if not. 

Given $w_{0} \in L^2 (-1,1)$, for each $\varepsilon>0$, let $w^{\varepsilon}$ be the solution to the evolution problem associated with the energy $E^{\varepsilon}$, and $w$ be the solution associated to the functional $E$, considering the same initial condition.
 
 \begin{theorem}
 Under the above assumptions, the solutions to the rescaled problems, $w^{\varepsilon}$, converge to $w$, the solution of \eqref{localcomplete}. 
 For any finite $T>0$ we have 
 \begin{equation}
     \lim_{\varepsilon \to 0}\left(\max_{t \in [0,T]}\parallel w^{\varepsilon}(\cdot,t)-w(\cdot,t) \parallel_{L^2(-1,1)}\right)=0.
 \end{equation}
 \end{theorem}
 \begin{proof}
 To prove this result we will make use of the Brezis-Pazy Theorem (Theorem A.37, see \cite{andreu2010nonlocal}), for the sequence of m-accretive operators $B_{J^{\varepsilon}} \in L^2(-1,1)$ defined in the previous section. To apply this result we would like to show the convergence of the resolvents, that is, we want to show that
 \begin{equation}\label{convergenceofresolvent}
     \lim_{\varepsilon \to 0}\left(I+B_{J^{\varepsilon}}\right)^{-1} \phi=\left(I+A\right)^{-1} \phi,
 \end{equation}
 where $A(w):=-w_{xx}$ is the classic operator for the heat equation, and for every $\phi \in L^2(-1,1)$. If we can prove \eqref{convergenceofresolvent} then, by the Brezis-Pazy Theorem, we get the convergence of the solutions $w^{\varepsilon}$ to $w$ in $L^2(-1,1)$ uniformly in $[0,T]$. To prove the convergence of resolvents, we will use a convergence result given by Mosco, checking the following statements:
 \begin{itemize}
     \item [1)] For every $w \in D(E)$ there exists a sequence $\{w^{\varepsilon}\} \in D(E^{\varepsilon})$ such that $w^{\varepsilon} \to w$ in $L^2(-1,1)$ and
     \begin{equation}
         E(w) \geq \limsup_{\varepsilon \to 0}{E^{\varepsilon}(w^{\varepsilon})}.
     \end{equation}
      \item [2)] If $w^{\varepsilon} \to w$ weakly in $L^2(-1,1)$ and
     \begin{equation}
         E(w) \leq \liminf_{\varepsilon \to 0}{E^{\varepsilon}(w^{\varepsilon})}.
     \end{equation}
          \end{itemize}
          
Let us start to prove the assertion $2)$. We can suppose that the limit infimum is finite, otherwise, there is nothing to prove. 
Hence, we can assume that $E^{\varepsilon}(w^{\varepsilon}) \leq C$. With this in mind and because all the terms involved in the  energy are positive, we have
\begin{itemize}
    \item [i)] $\displaystyle \frac{1}{2}\int_{-1}^{0} \Big| \frac{\partial u^{\varepsilon}}{\partial x}\Big|^2 dx \leq C$;
    \item [ii)] $\displaystyle \frac{C_{J,1}}{4\varepsilon^{3}}\int_{0}^{1}\int_{0}^{1}J^{\varepsilon}(x-y) (v^{\varepsilon}(y)-v^{\varepsilon}(x))^2 dydx \leq C$;
    \item [iii)] $\displaystyle \frac{C_{J,2}}{2\varepsilon^{3}}\int_{-1}^{0}\int_{0}^{1}J^{\varepsilon}(x-y) (u^{\varepsilon}(0)-v^{\varepsilon}(x))^2 dxdy \leq C$.
\end{itemize}
From $i)$, it follows that, there exists a subsequence, also denoted by $\{u^{\varepsilon}\}$, such that 
\begin{equation}\label{uconv1}
   u^{\varepsilon} \rightharpoonup u \quad \in H^1(-1,0),
\end{equation}
which implies
\begin{equation}\label{uconv12}
   u^{\varepsilon} \to u \quad \mbox{ in }L^2(-1,0),
\qquad \mbox{and} \qquad
   u^{\varepsilon} \to u \quad \text{uniformly in } (-1,0).
\end{equation}

We also know that 
\begin{align}\nonumber
 & \frac{C_{J,2}}{2\varepsilon^{3}}\int_{-1}^{0}\int_{0}^{1}J^{\varepsilon}(x-y) (u^{\varepsilon}(y)-v^{\varepsilon}(x))^2 dxdy \\ \nonumber
 & \leq \frac{C_{J,2}}{2\varepsilon^{3}}\int_{-1}^{0}\int_{0}^{1}J^{\varepsilon}(x-y) (u^{\varepsilon}(y)-u^{\varepsilon}(0))^2 dxdy \\\nonumber
 & \qquad + \underbrace{\frac{C_{J,2}}{2\varepsilon^{3}}\int_{-1}^{0}\int_{0}^{1}J^{\varepsilon}(x-y) (u^{\varepsilon}(0)-v^{\varepsilon}(x))^2 dxdy.}_{\leq C}
\end{align}

Let us show that $$\frac{C_{J,2}}{2\varepsilon^{3}}\int_{-1}^{0}\int_{0}^{1}J^{\varepsilon}(x-y) (u^{\varepsilon}(y)-u^{\varepsilon}(0))^2 dxdy$$ is bounded. After a change of variables and observing that the $supp(J)=B(0,R)$, we get
\begin{align*}
  & \frac{C_{J,2}}{2\varepsilon^{3}}\int_{-1}^{0}\int_{0}^{1}J^{\varepsilon}(x-y) (u^{\varepsilon}(y)-u^{\varepsilon}(0))^2 dxdy \\[6pt]
  & =  \frac{C_{J,2}}{2\varepsilon^{2}}\int_{-1}^{0}\int_{\frac{-y}{\varepsilon}}^{\frac{1-y}{\varepsilon}}J(z)dz (u^{\varepsilon}(y)-u^{\varepsilon}(0))^2 dxdy 
  \\[6pt]
   & = \frac{C_{J,2}}{2\varepsilon^{2}}\int_{-R\varepsilon}^{0}f_{\varepsilon}(y)\frac{u^{\varepsilon}(y)-u^{\varepsilon}(0))^2}{\varepsilon}
   \frac{dy}{\varepsilon}.
\end{align*}
Changing variables again, using Holder's inequality and the arithmetic-geometric inequality, it follows that 
\begin{align}\label{auxinequality}\nonumber
  & \frac{C_{J,2}}{2}\int_{-R}^{0}f_{\varepsilon}(\varepsilon w)\frac{1}{\varepsilon}\left((u^{\varepsilon}(0)-u^{\varepsilon}(\varepsilon w)\right)^2 dw \\\nonumber
  & \leq \frac{C_{J,2}}{2}\int_{-R}^{0}f_{\varepsilon}(\varepsilon w) \left[ (-w^2)^{1/2}\left(\int_{\varepsilon w}^{0}(u^{\varepsilon}_{x}(s))^2ds\right)^{1/2}\right]dw \\\nonumber
  &\leq \frac{C_{J,2}}{2}\int_{-R}^{0}f_{\varepsilon}(\varepsilon w) \left[ \frac{1}{2}(-w^{2})+ \frac{1}{2}\int_{\varepsilon w}^{0}(u^{\varepsilon}_{x}(s))^2ds\right]dw \\\nonumber
  & \leq \frac{C_{J,2}}{8}\int_{-R}^{0}(-w^{2})dw + \frac{C_{2}}{8}\int_{-R}^{0}\left[\int_{\varepsilon w}^{0}(u^{\varepsilon}_{x}(s))^2ds\right] dw \\ 
  & = \Tilde{C} + \frac{C_{J,2}}{8}\int_{-R}^{0}\underbrace{\left[\int_{-1}^{0}(u^{\varepsilon}_{x}(s))^2ds\right]}_{\leq C} dw .
\end{align}
Therefore, we conclude that $$\frac{C_{J,2}}{2\varepsilon^{3}}\int_{-1}^{0}\int_{0}^{1}J^{\varepsilon}(x-y) (u^{\varepsilon}(y)-v^{\varepsilon}(x))^2 dxdy$$ is bounded. By \eqref{auxinequality} we obtain 
\begin{align}
      & \nonumber \overline{E}(w^{\varepsilon}) :=\frac{1}{2}\int_{-1}^{0}\Big| \frac{\partial u^{\varepsilon}}{\partial x}\Big|^2 dx + \frac{C_{J,1}}{4\varepsilon^{3}}\int_{0}^{1}\int_{0}^{1}J^{\varepsilon}(x-y)\left(v^{\varepsilon}(y)-v^{\varepsilon}(x)\right)^2 dydx
      \\[10pt] 
      \nonumber
   & \qquad \qquad \qquad + \frac{C_{J,2}}{2\varepsilon^{3}}\int_{-1}^{0}\int_{0}^{1}J^{\varepsilon}(x-y)\left(u^{\varepsilon}(y)-v^{\varepsilon}(x)\right)^2 dxdy \leq C.
\end{align}
By Lemma \ref{energy.lema}, there exists $k>0$ (independent of $\varepsilon$) such that
\begin{equation}\label{important}
    C \geq \overline{E}(w^{\varepsilon}) \geq k \frac{1}{\varepsilon^{3}}\int_{-1}^{1}\int_{-1}^{1}J^{\varepsilon}(x-y)\left(w^{\varepsilon}(y)-w^{\varepsilon}(x)\right)^2 dy dx.
\end{equation}
It follows that, there exists a subsequence, also denoted $w^{\varepsilon}$, which converges in $L^2(-1,1)$ to a limit $w \in H^1(-1,1)$. 

We have
\begin{equation}\label{liminf1}
    \liminf_{\varepsilon \to 0} \frac{1}{2}\int_{-1}^{0} \Big| \frac{\partial u^{\varepsilon}}{\partial x}\Big|^2
    dx \geq \frac{1}{2}\int_{-1}^{0} \Big| \frac{\partial u}{\partial x}\Big|^2 dx.
\end{equation}

Using the fact that $$\frac{C_{J,2}}{2\varepsilon^{3}}\int_{0}^{1}\int_{0}^{1}J^{\varepsilon}(x-y) (v^{\varepsilon}(y)-v^{\varepsilon}(x))^2 dxdy$$ is bounded, by Theorem 6.11 in \cite{andreu2010nonlocal}, there exists a subsequence, also denoted by $\{v^{\varepsilon}\}$, such that 
\begin{equation*}
    v^{\varepsilon} \to v \quad \text{in} \quad L^2(0,1),
\end{equation*}
and, moreover, the limit $v$ satisfies, $v \in H^{1}(0,1)$ and 
\begin{equation}\label{lemma2.4}
    \left(\frac{C_{J,2}}{4}J(z)\right)^{1/2}\frac{\overline{v}^{\varepsilon}(x+\varepsilon z)-v^{\varepsilon}(x)}{\varepsilon} \rightharpoonup \left(\frac{C_{J,2}}{4}J(z)\right)^{1/2} z \cdot \frac{\partial v}{\partial x},
\end{equation}
weakly in $L^2(0,1)\times L^2(\mathbb{R})$. Then, taking the limit in the equation \eqref{lemma2.4} we have that
\begin{equation}\label{liminf2}
 \liminf_{\varepsilon \to 0}{\frac{C_{J,1}}{4\varepsilon^3}\int_{0}^{1}\int_{0}^{1}J^{\varepsilon}(x-y)\left(v^{\varepsilon}(y)-v^{\varepsilon}(x)\right)^2 dydx} \geq \frac{1}{2}\int_{0}^{1}\Big| \frac{\partial v}{\partial x}\Big|^2 dx.
\end{equation}

Moreover, we have
\begin{equation}\label{liminf3}
 \liminf_{\varepsilon \to 0}{\frac{C_2}{2\varepsilon^3}\int_{-1}^{0}\int_{0}^{1}J^{\varepsilon}(x-y)\left(v^{\varepsilon}(y)-u^{\varepsilon}(0)\right)^2 dxdy} \geq 0.
\end{equation}

Therefore, from \eqref{liminf1}-\eqref{liminf3} we conclude that
\begin{equation*}
   \liminf_{\varepsilon \to 0}{E^{\varepsilon}(w^{\varepsilon})} \geq  \frac{1}{2}\int_{-1}^{0}
   \Big| \frac{\partial u}{\partial x}\Big|^2 dx + \frac{1}{2}\int_{0}^{1} \Big| \frac{\partial v}{\partial x}\Big|^2 dx = E(w).
\end{equation*}

Now let us prove $1)$. Given $w \in H^1(-1,1)$ we choose as the approximating sequence $w_{n} \equiv w$. 
We have,
\begin{align*}
  & E^{\varepsilon}(w)  :=\frac{1}{2}\int_{-1}^{0}\Big| \frac{\partial w}{\partial x}\Big|^2 dx + \frac{C_{J,1}}{4\varepsilon^{3}}\int_{0}^{1}\int_{0}^{1}J^{\varepsilon}(x-y)\left(w(y)-w(x)\right)^2 dydx \\[10pt]
   \nonumber
   & \qquad + \frac{C_{J,2}}{2\varepsilon^{3}}\int_{-1}^{0}\int_{0}^{1}J^{\varepsilon}(x-y)\left(w(0)-w(x)\right)^2 dxdy 
\end{align*}
and we want to show that 
\begin{equation}\label{inequalityE}
    \limsup_{\varepsilon \to 0}{ E^{\varepsilon}(w)} \leq E(w).
\end{equation}

The inequality \eqref{inequalityE} is hold if we can verify that
    \begin{equation}\label{des1}
        \limsup_{\varepsilon \to 0}\left(\frac{C_{J,1}}{4\varepsilon^{3}}\int_{0}^{1}\int_{0}^{1}J^{\varepsilon}(x-y)\left(w(y)-w(x)\right)^2 dydx\right)
        =\frac{1}{2}\int_{0}^{1}\Big| \frac{\partial w}{\partial x}\Big|^2
        dx.
    \end{equation}
    and
     \begin{equation}\label{des2}
        \limsup_{\varepsilon \to 0}\left(\frac{C_{J,2}}{2\varepsilon^{3}}\int_{-1}^{0}\int_{0}^{1}J^{\varepsilon}(x-y)\left(w(0)-w(x)\right)^2 dxdy\right)=0.
    \end{equation}

Let us first show \eqref{des2}. Performing a change of variables and using the Holder's inequality, \eqref{des2} can be written as
\begin{align*}\nonumber
  & \frac{C_{J,2}}{2\varepsilon^{3}}\int_{-1}^{0}\int_{0}^{1}J^{\varepsilon}(x-y)\left(w(0)-w(x)\right)^2 dxdy \\
  & = \frac{C_{J,2}}{2\varepsilon^{2}}\int_{-1}^{0}\int_{\frac{-y}{\varepsilon}}^{\frac{1-y}{\varepsilon}}J(z)\left(w(y+\varepsilon z)-w(0)\right)^2 dzdy \\\nonumber
   & = \frac{C_{J,2}}{2\varepsilon^{2}}\int_{-R\varepsilon}^{0}\int_{\frac{-y}{\varepsilon}}^{\frac{1-y}{\varepsilon}}J(z)\left(w(y+\varepsilon z)-w(0)\right)^2 dzdy \\\nonumber
   & = \frac{C_{J,2}}{2}\int_{-R\varepsilon}^{0}\int_{\frac{-y}{\varepsilon}}^{R}J(z)\left[\int_{0}^{y+\varepsilon z}
   \frac{ \frac{\partial w}{\partial x}(s)}{\varepsilon}ds\right]^2 dzdy \\\nonumber
   & \leq \frac{C_{J,2}}{2}\int_{-R\varepsilon}^{0}\int_{\frac{-y}{\varepsilon}}^{R}J(z)\left[\int_{0}^{y+\varepsilon z}
   \Big|
   \frac{\partial w}{\partial x}(s)\Big|^{2}ds\right] dz\frac{dy}{\varepsilon}.
\end{align*}
Changing variables again and since $\int_{-R}^{R}J(z)dz=1$, we obtain
\begin{align*}\nonumber
& \frac{C_{J,2}}{2}\int_{-R}^{0}\int_{-t}^{R}J(z)\left[\int_{0}^{\varepsilon(t+z)} \Big|\frac{\partial w}{\partial x}\Big|^{2}ds\right] dz dt \\
& \leq \frac{C_{J,2}}{2}\int_{-R}^{0}\int_{-R}^{R}J(z)dz\left[\int_{0}^{2R\varepsilon}\Big|\frac{\partial w}{\partial x}\Big|^{2} ds\right] dt \\ \nonumber
&\leq \frac{C_{J,2}}{2}\int_{-R}^{0} \left[\int_{0}^{2R\varepsilon} \Big|\frac{\partial w}{\partial x}\Big|^{2} ds\right] dt.
\end{align*}
Now, we observe that, as $\frac{\partial w}{\partial x} \in L^2(-1,1)$ then $|\frac{\partial w}{\partial x}|^{2} \in L^1(-1,1)$. Then, 
we have $$\int_{0}^{2R\varepsilon}\Big|\frac{\partial w}{\partial x}\Big|^{2} dz \to 0$$ as $\varepsilon \to 0$, which yields \eqref{des2}.

Now we have left with the task to show that
 \begin{equation*}
        \limsup_{\varepsilon \to 0}\left(\frac{C_{J,1}}{4\varepsilon^{3}}\int_{0}^{1}\int_{0}^{1}J^{\varepsilon}(x-y)\left(w(y)-w(x)\right)^2 dydx\right)=\frac{1}{2}\int_{0}^{1} \Big|\frac{\partial w}{\partial x}\Big|^{2} dx.
    \end{equation*}
Changing variables and using Taylor's expansion it follows that
\small
\begin{align*}
   & \left|\frac{C_{J,1}}{4\varepsilon^{3}}\int_{0}^{1}\int_{0}^{1}J^{\varepsilon}(x-y)\left(w(y)-w(x)\right)^2 dydx\right|
   \\[10pt]
   & =  \left|\frac{C_{J,1}}{4\varepsilon^{2}}\int_{0}^{1}\int_{\frac{-x}{\varepsilon}}^{\frac{1-x}{\varepsilon}}J(z)\left(w(x+\varepsilon z)-w(x)\right)^2 dzdx\right| 
   \\[10pt]
    & \leq \frac{C_{J,1}}{4\varepsilon^{2}}\int_{0}^{1}\int_{\frac{-x}{\varepsilon}}^{\frac{1-x}{\varepsilon}}J(z)\left|w(x+\varepsilon z)-w(x)\right|^2 dzdx 
    \\[10pt]
    & = \frac{C_{J,1}}{4\varepsilon^{2}}\int_{0}^{1}\int_{\frac{-x}{\varepsilon}}^{\frac{1-x}{\varepsilon}}J(z)
    \left| \frac{\partial w}{\partial x}(x)\varepsilon z+\frac{1}{2} \frac{\partial^2 w}{\partial x^2}(\xi)\varepsilon^2 z^2\right|^2 dzdx 
    \\[10pt]
    & \leq \frac{C_{J,1}}{4\varepsilon^{2}}\int_{0}^{1}\int_{\frac{-x}{\varepsilon}}^{\frac{1-x}{\varepsilon}}J(z)
    \left(\left| \frac{\partial w}{\partial x}(x)\varepsilon z \right|+ \frac{1}{2} \left|\frac{\partial^2 w}{\partial x^2}(\xi)\varepsilon^2 z^2\right|
    \right)^2 dzdx.
\end{align*}
\normalsize
Now, using Minkowski's inequality
\begin{align*}
 & \frac{C_{J,1}}{4\varepsilon^{2}}\int_{0}^{1}\int_{\frac{-x}{\varepsilon}}^{\frac{1-x}{\varepsilon}}J(z)
     \left(\left| \frac{\partial w}{\partial x}(x)\varepsilon z \right|+ \frac{1}{2} \left|\frac{\partial^2 w}{\partial x^2}(\xi)\varepsilon^2 z^2\right|
    \right)^2 dzdx \\[10pt]
    & \leq \frac{C_{J,1}}{4\varepsilon^{2}}\int_{0}^{1}\int_{\frac{-x}{\varepsilon}}^{\frac{1-x}{\varepsilon}}J(z)
    \left|\frac{\partial w}{\partial x}(x)\varepsilon z \right|^2 dzdx 
    +\frac{C_{J,1}}{4\varepsilon^{2}}\int_{0}^{1}\int_{\frac{-x}{\varepsilon}}^{\frac{1-x}{\varepsilon}}J(z)
    \left|\frac{1}{2} \frac{\partial^2 w}{\partial x^2}(\xi)\varepsilon^2 z^2\right| dzdx \\[10pt]
    & \leq \frac{C_{J,1}}{4}\int_{0}^{1}\int_{\frac{-x}{\varepsilon}}^{\frac{1-x}{\varepsilon}}J(z)
    \left| \frac{\partial w}{\partial x} (x)\right|^2 |z|^2 dzdx  +  \varepsilon^2 \frac{C_{J,1}}{16}\int_{0}^{1}\int_{\frac{-x}{\varepsilon}}^{\frac{1-x}{\varepsilon}}J(z)
    \left|\frac{\partial^2 w}{\partial x^2}(\xi)\right|^2 |z|^4 dzdx \\[10pt]
    & \leq \frac{C_{J,1}}{4}\int_{0}^{1}\int_{-R}^{R}J(z) \left|\frac{\partial w}{\partial x}(x)\right|^2 |z|^2 dzdx 
    +  \varepsilon^2 \frac{C_{J,1}}{16}\int_{0}^{1}\int_{-R}^{R}J(z)\left|\frac{\partial^2 w}{\partial x^2}(\xi)\right|^2 |z|^4 dzdx \\[10pt]
    & \leq \frac{C_{J,1}}{4}\int_{0}^{1}\int_{\mathbb{R}}J(z) |z|^2 dz \left|\frac{\partial w}{\partial x}(x)\right|^2dx 
     +  \varepsilon^2 \frac{C_{J,1}}{16}\int_{0}^{1}\int_{\mathbb{R}}J(z)\left|\frac{\partial^2 w}{\partial x^2}(\xi)\right|^2 |z|^4 dzdx.
\end{align*}

Since $ \int_{\mathbb{R}}J(z) |z|^2 dz = M(J)$, $\frac{\partial^2 w}{\partial x^2}$ is bounded and $\int_{\mathbb{R}}J(z)|z|^4 dz$ is finite 
we can conclude that
\begin{align*}
  & \limsup_{\varepsilon \to 0}\left(\varepsilon^2 \frac{C_{J,1}}{16}\int_{0}^{1}\int_{\mathbb{R}}J(z)|z|^4 dz\left|\frac{\partial^2 w}{\partial x^2}(\xi)\right|^2 dx\right) = 0, \\[10pt]
 &  \text{and} \\[10pt]
 & \limsup_{\varepsilon \to 0}\left(\frac{C_{J,1}. M(J)}{4}\int_{0}^{1}\left|\frac{\partial w}{\partial x}(x)\right|^2 |z|^2 dx\right) = \frac{1}{2}\int_{0}^{1}
 \Big| \frac{\partial w}{\partial x} \Big|^2dx.
\end{align*}

Finally, we have
\begin{align*}
  &  \limsup_{\varepsilon \to 0}{ E^{\varepsilon}(w)} \\
  & =\limsup_{\varepsilon \to 0} \left(\frac{1}{2}\int_{-1}^{0} \Big|\frac{\partial w}{\partial x}\Big|^2 dx 
  + \frac{C_{J,1}}{4\varepsilon^{3}}\int_{0}^{1}\int_{0}^{1}J^{\varepsilon}(x-y)\left(w(y)-w(x)\right)^2 dydx\right. \\[10pt] 
   & \qquad + \left.\frac{C_{J,2}}{2\varepsilon^{3}}\int_{-1}^{0}\int_{0}^{1}J^{\varepsilon}(x-y)\left(w(0)-w(x)\right)^2 dxdy\right) \\[10pt]
   &= \limsup_{\varepsilon \to 0}\left(\frac{1}{2}\int_{-1}^{0} \Big|\frac{\partial w}{\partial x}\Big|^2  dx \right) \\[10pt]
   & \qquad + \limsup_{\varepsilon \to 0}\left(\frac{C_{J,1}}{4\varepsilon^{3}}\int_{0}^{1}\int_{0}^{1}J^{\varepsilon}(x-y)\left(w(y)-w(x)\right)^2 dydx \right) 
   \\[10pt]
   & \qquad + \limsup_{\varepsilon \to 0}\left(\frac{C_{J,2}}{2\varepsilon^{3}}\int_{-1}^{0}\int_{0}^{1}J^{\varepsilon}(x-y)\left(w(0)-w(x)\right)^2 
   dxdy\right)\\[10pt]
   & \leq \frac{1}{2}\int_{-1}^{0}\Big|\frac{\partial w}{\partial x}\Big|^2  dx + \frac{1}{2}\int_{0}^{1} \Big|\frac{\partial w}{\partial x}\Big|^2 dx \\[10pt]
   & = E(w),
\end{align*}
as we wanted to show.
 \end{proof}

 \begin{remark} {\rm Our convergence result can be also read as: take, as before,
 $w^{\varepsilon} = (u^{\varepsilon}, v^{\varepsilon})$. Then, 
 for any finite $T>0$ we have 
 \begin{equation}
     \lim_{\varepsilon \to 0}\left(\max_{t \in [0,T]}\parallel u^{\varepsilon}(\cdot,t)-u(\cdot,t) \parallel_{L^2(-1,1)}\right)=0
 \end{equation}
 and 
 \begin{equation}
     \lim_{\varepsilon \to 0}\left(\max_{t \in [0,T]}\parallel v^{\varepsilon}(\cdot,t)-v(\cdot,t) \parallel_{L^2(-1,1)}\right)=0.
 \end{equation}
 The limit pair $(u,v)$ is the unique solution to two heat equations
 \begin{align}\label{local.limite.uv}
\begin{cases}
 \displaystyle   \frac{\partial u}{\partial t}(x,t)  =  \frac{\partial^{2}u }{\partial x^{2}}(x,t), \qquad x\in (-1,0), t\in (0,T) ,\\[10pt]
 \displaystyle   \frac{\partial u}{\partial x}(-1,t)  =  0,  
     \end{cases}
    \end{align}
 \begin{align}\label{local.limite.uv2}
\begin{cases}
 \displaystyle \frac{\partial v}{\partial t}(x,t)  =  \frac{\partial^{2}v }{\partial x^{2}}(x,t),   \qquad x\in (0,1), t\in (0,T),\\[10pt]
 \displaystyle   \frac{\partial v}{\partial x}(1,t)  =  0,  
    \end{cases}
    \end{align}
    with coupling 
    $$
    u(0,t) = v(0,t), \qquad
 \displaystyle    \frac{\partial u}{\partial x}(0,t)  =  \frac{\partial v}{\partial x}(0,t)  
    $$
and initial conditions 
$$
 u(x,0)  =u_{0}(x), \qquad v(x,0)  =v_{0}(x).
$$
Notice that the coupling gives continuity and continuity of the derivative of the function
 \begin{align}\label{complete.55}
w(x,t) = 
\begin{cases}
   u(x,t), \quad \text{if} \quad x \in (-1,0) \\
    v(x,t), \quad \text{if} \quad x \in (0,1) 
    \end{cases}
    \end{align}
that therefore turns out to be a solution to 
 \begin{align}\label{localcomplete.909}
\begin{cases}
  \displaystyle  \frac{\partial w}{\partial t}  (x,t)= \frac{\partial^2 w}{\partial x^2}(x,t), \quad x \in (-1,1), \, t>0,\\[10pt]
   \displaystyle \frac{\partial w}{\partial x}(-1,t)=\frac{\partial w}{\partial x}(1,t)=0, \quad t>0,\\[10pt]
    w(x,0)=w_{0}(x), \quad x \in (-1,1).
    \end{cases}
\end{align}
 }
 \end{remark}

\section{Extension to higher dimensions.} \label{sect-higher} 
In this final section, we will briefly describe how our result can be extended
to higher dimensions.
Take $\Omega$, as a bounded smooth domain in $\mathbb{R}^N$ and split it into two subdomains $\Omega_l$ and $\Omega_{nl}$, $\Omega = \Omega_l \cup \Omega_{nl}$. Let us call $\Sigma$, the interface 
between  $\Omega_l$ and $\Omega_{nl}$ inside $\Omega$, that is, 
$$
\Sigma =  \overline{\Omega_l} \cap \overline{\Omega_{nl}} \cap \Omega.
$$
We will assume that $\Omega_l$ has Lipschitz boundary (in order to solve a heat equation with Newman boundary conditions, we need some regularity of the boundary).
We will also assume the following geometric condition on the interface $\Sigma$; foe every $x \in \Omega_l$ and every
$y\in \Omega_{nl}$ with $x-y  \in \mbox{supp} (J) $ there exists a unique $z\in \Sigma$ that belongs to the segment that joins 
$x$ with $y$.  To provide examples, notice that this geometric condition holds if $\Sigma$ is almost flat.
This assumption is useful since, from a probabilistic viewpoint, when a particle wants to jump
from $y \in \Omega_{nl}$ to $x \in \Omega_{l}$ we want that it gets stuck at the interface (and then we want that there exist a unique point on $\Sigma$ that belongs to the segment $[x,y]$, otherwise, some selection principle has to
be assumed and, the selected point on the interface will not depend continuously on $x$ and $y$, in general).
This assumption will be also helpful to make a variable change in the coupling term that appears in our energy (see below).

As before, we split $w\in L^2 (\Omega)$ as $w =u+v$, with $u =w \chi_{\Omega_l}$
and $v = w \chi_{\Omega_{nl}}$. For any $$w=(u,v) \in \mathcal{B}:=\left\{w \in L^2(\Omega): u|_{\Omega_{l}} \in H^1(\Omega_{l}),  
v \in L^2(\Omega_{nl}) \right\}$$ we define the energy
\begin{equation} \label{energy.909}
\begin{array}{l}
    E(u,v)  \displaystyle:=\frac{1}{2}\int_{\Omega_{l}} | \nabla u |^2 dx + \frac{C_{J,1}}{4}\int_{\Omega_{nl}}\int_{\Omega_{nl}}J(x-y)\left(v(y)-v(x)\right)^2 dydx \\[10pt]
    \displaystyle \quad \qquad \qquad + \frac{C_{J,2}}{2}\int_{\Omega_{nl}}\int_{\Sigma}G(x,z)\left(v(x)-u(z)\right)^2 d\sigma(z) dx.
\end{array}
\end{equation}
Remark that in this energy we have 
$$
\int_{\Omega_{nl}}\int_{\Sigma}G(x,z)dy\left(v(x)-u(z)\right)^2 d\sigma(y) dx
$$
as coupling term. This integral comes from the integral
$$
\iint_{A}J(x-y) \left(v(x)-u(z)\right)^2 dy dx
$$
with $A=\{ (x,y) : x \in \Omega_{nl}, y \in \Omega_l, \mbox{ with } z \in \Sigma, z=ax+(1-a)y \}$
(that is $z$ lies in the segment that joins $x$ and $y$) after a change of variables (just take 
$z=ax+(1-a)y$). It is here (in doing this change of variables) that we are using the
geometric condition on $\Sigma$). The kernel $G$ is also nonnegative but is not equal to $J$ (this fact comes
from the change of variables that involves a jacobian $D(x,z)$).

With this energy, the associated evolution problems read as,
 \begin{align}\label{local-N}
\begin{cases}
 \displaystyle   \frac{\partial u}{\partial t}(x,t)  =  \Delta u (x,t),  \\[10pt]
 \displaystyle   \frac{\partial u}{\partial \eta}(z,t)  =  0, \qquad z\in \partial \Omega_l \cap \partial \Omega,  \\[10pt]
 \displaystyle    \frac{\partial u}{\partial \eta}(z,t)  = C_{J,2}\int_{\Omega_{nl}} G(x,z)( v(y,t)-  u(z,t)) dx, \qquad z \in \Sigma, \\[10pt] 
 \displaystyle    u(x,0)  =u_{0}(x).
    \end{cases}
    \end{align}
    for $ x \in \Omega_l $, $t>0$, and
\begin{align}\label{nonlocal-N}
\begin{cases}
 \displaystyle \frac{\partial v}{\partial t}(x,t)  =C_{J,1}\int_{\Omega_{nl}} J(x-y)\left(v(y,t)-v(x,t) \right)dy 
   - C_{J,2} \int_\Sigma G(x,z) ( v(x,t)-  u(z,t)) d\sigma (z) ,  \\[10pt]
     v(x,0)  =v_{0}(x), 
      \end{cases}
     \end{align} 
     for $ x \in \Omega_{nl}$, $t>0$. 
     
     For this problem \eqref{local-N}--\eqref{nonlocal-N}, we can also prove existence and uniqueness 
     following the same steps that we made for the one-dimensional case. In fact, the strategy
     of building a solution as a fixed point of the composition of the maps that solves the 
     problem for $u$ (given $v$) and for $v$ (fixing $u$) also works here. Remark that we
     obtain a solution $u(x,t) $ that is in $H^1 (\Omega_{l})$ for $t>0$ and hence $u(z,t)$
     is defined on $\Sigma$ in the sense of traces (and belongs to $L^2 (\Sigma)$ for $t>0$).
     The more abstract approach using semigroup theory also works here. Consider the operator
\begin{align*}
 \displaystyle 
B_{J}(u,v) =
\begin{cases} 
 - \Delta u \qquad  \text{for} \quad x \in \Omega_l,\\[10pt]
\displaystyle  -\frac{C_{J,1}}{2}\int_{\Omega_{nl}}J(x-y)(u(y)-u(x))dy+C_{J,2}\int_{\Sigma}G(x,z)(v(x)-u(z))d\sigma (z) 
\qquad \text{for} \quad x \in \Omega_{nl},
\end{cases}
\end{align*}
with domain $$
\begin{array}{l} \displaystyle
D(B_{J}):=  \Big \{ (u,v) : u \in H^2(\Omega_{l}, v\in L^2(\Omega),
\mbox{ with } \frac{\partial u}{\partial \eta}(z) = 0 \mbox{ on }\partial \Omega \cap \partial \Omega_l  \\[10pt]
\displaystyle \qquad \qquad \qquad \qquad 
\mbox{ and } \frac{\partial u}{\partial \eta}(z) = -C_{J,2}\int_{\Omega_{nl} }G(x,z)(v(x)-u(z))dx \mbox{ on  } \Sigma,
\Big\}
\end{array}
$$
and proceed as we did previously.

     The total mass is preserved. In fact, we have
 \begin{align*}
\displaystyle 
 \frac{\partial}{\partial {t}}\left(\int_{\Omega}w(x,t)dx\right) &
 \displaystyle = \int_{\Omega_l} \Delta u(x,t)dx+C_{J,1}\int_{\Omega_{nl}}\int_{\Omega_{nl}}J(x-y)(v(y,t)-v(x,t))dydx \\[6pt]
& \qquad -C_{J,2}\int_{\Omega_{nl}}\int_{\Sigma }G(x,z)(v(x,t)- u(z,t) d\sigma (z) dx \\[6pt]
& = \int_{\partial \Omega_l} \frac{\partial u}{\partial \eta}(x,t)dx  -C_{J,2}\int_{\Omega_{nl}}\int_{\Sigma }G(x,z)(v(x,t)- u(z,t) d\sigma (z) dx \\[6pt]
& = 0.
\end{align*}

     The key control of the nonlocal energy,
     \begin{equation} \label{energycontrol-N}
\begin{array}{l}
\displaystyle
 \frac{1}{2} \int_{\Omega_l}| \nabla u|^2 dx + \frac{C_{J,1}}{4}\int_{\Omega_{nl}}\int_{\Omega_{nl}}J(x-y)\left(v(y)-v(x)\right)^2 dydx 
  + \frac{C_{J,2}}{2}\int_{\Omega_{nl}}\int_{\Sigma}J(x,z)\left(v(x)-u(z)\right)^2 d\sigma (z) dx  \\[10pt]
  \displaystyle \qquad  \geq k \int_{\Omega}\int_{\Omega}J(x-y)\left(w(y)-w(x)\right)^2 dydx.
\end{array}
\end{equation}
can be proved, as before, arguing by contradiction. 

With the key inequality \eqref{energycontrol-N}, we can show that solutions converge to the mean value of the initial condition, as $t \to \infty$ with an exponential rate.
\begin{equation}
    \left\| w(\cdot ,t) - \fint w_0 \right\|_{L^2(-1,1)} \leq C e^{-\beta_{1} t}, \qquad t>0.
\end{equation}
In fact, we have that
\begin{align}\label{eigenvalue-N}
\displaystyle 
    0<\beta_{1}=\inf_{w : \int_{\Omega} w=0}\frac{
    \displaystyle E(w)}{\displaystyle  \int_{\Omega}(w(x))^2 dx}
\end{align}
is strictly positive. This fact can be proved by contradictions as we did before, but it also follows from
\eqref{energycontrol-N} and the results in \cite{andreu2010nonlocal} since we have
$$
\displaystyle 
   \beta_{1}=\inf_{w : \int_{\Omega} w=0}\frac{
    \displaystyle E(w)}{\displaystyle  \int_{\Omega}(w(x))^2 dx} \geq
    \inf_{w : \int_{\Omega} w=0}\frac{
    \displaystyle k \int_{\Omega}\int_{\Omega}J(x-y)\left(w(y)-w(x)\right)^2 dydx}{\displaystyle  \int_{\Omega}(w(x))^2 dx} >0.
$$
 
 The approximation of the heat equation with Neumann boundary conditions under rescales of the kernel is left open. We
 believe that the result holds with extra assumptions on the coupling kernel $G$.

\medskip

\noindent{\large \textbf{Acknowledgments}}

BCS was financed by the Coordenação de Aperfeiçoamento de Pessoal de Nível Superior - Brasil (Capes) - No 88887369814/2019-00.

JDR is partially supported by CONICET grant PIP GI No 11220150100036CO
(Argentina), by  UBACyT grant 20020160100155BA (Argentina) and by the Spanish project MTM2015-70227-P.


{\bf addresses}

B. C. dos Santos and S. Oliva. \\
IME-USP \\
  Institute of Mathematics and Statistics\\
  University of São Paulo, Brazil \\

J. D. Rossi \\ 
 Department of Mathematics, FCEyN\\
  University of Buenos Aires, Argentina

\end{document}